\documentclass[a4paper,11pt, oneside]{article}
\usepackage[table,xcdraw]{xcolor}
\usepackage[english]{babel}
\usepackage[latin1,utf8]{inputenx}
\usepackage{alltt}
\usepackage{amssymb}
\usepackage{latexsym}
\usepackage{graphicx}
\usepackage{mathrsfs}
\usepackage{epsfig}
\usepackage[colorlinks=true,
            linkcolor=black,
            citecolor=black]{hyperref} 
\usepackage{adjustbox}
\usepackage[hypcap]{caption}
\usepackage{pdfpages}
\usepackage{amsthm}
\usepackage{amsmath}
\usepackage{float}
\usepackage{titlesec}
\usepackage{setspace}
\usepackage{subcaption}

\usepackage{dsfont}
\usepackage[square,numbers]{natbib}
\newcommand{\I}{{\mathbb I}}
\newcommand{\C}{\mathcal{C}} 
\newcommand{\R}{\mathbb{R}} 
\newcommand{\N}{\mathbb{N}} 

\newcommand{\PP}{\mathbb{P}} 

\DeclareMathOperator*{\argmin}{arg\,min}

\usepackage{tikz}
\usetikzlibrary{positioning}
\usepackage{tikz-cd}
\usepackage{tkz-euclide}

\renewenvironment{proof}{\begin{Proof}[\bfseries\proofname]}{\end{Proof}}

\usepackage{array}
\newcolumntype{M}[1]{>{\centering\arraybackslash}m{#1}}

\newtheorem{Theorem}{Theorem}[section]
 
\newtheorem{Lemma}[Theorem]{Lemma}	
\theoremstyle{definition}
\newtheorem{Example}[Theorem]{Example}

\theoremstyle{definition}
\newtheorem{Remark}[Theorem]{Remark} 
\newtheorem{Proposition}[Theorem]{Proposition}

\allowdisplaybreaks

\numberwithin{equation}{section}

\begin{document}

\title{Total positivity of copulas from a Markov kernel perspective}
\author{S. Fuchs\footnote{{Department for Artificial Intelligence \& Human Interfaces, University of Salzburg, Hellbrunnerstrasse 34, 5020 Salzburg, Austria. sebastian.fuchs@plus.ac.at}}, M. Tschimpke\footnote{{Department for Artificial Intelligence \& Human Interfaces, University of Salzburg, Hellbrunnerstrasse 34, 5020 Salzburg, Austria. marco.tschimpke@plus.ac.at}}}
\date{\today}

\maketitle


\begin{abstract}
\noindent
The underlying dependence structure between two random variables can be described in manifold ways. 
This includes the examination of certain dependence properties such as lower tail decreasingness (LTD), stochastic increasingness (SI) or total positivity of order $2$,
the latter usually considered for a copula (TP2) or (if existent) its density (d-TP2).
In the present paper we investigate total positivity of order $2$ for a copula's Markov kernel (MK-TP2 for short), 
a positive dependence property that is stronger than TP2 and SI, weaker than d-TP2 but, unlike d-TP2, is not restricted to absolutely continuous copulas, 
making it presumably the strongest dependence property defined for any copula (including those with a singular part such as Marshall-Olkin copulas).
We examine the MK-TP2 property for different copula families,
among them the class of Archimedean copulas and the class of extreme value copulas.
In particular we show that, within the class of Archimedean copulas, the dependence properties SI and MK-TP2 are equivalent.
\end{abstract}


\section{Introduction}
Copulas capture the many facets of dependence relationships between (continuous) random variables, and various ways exist to quantify and describe them.
For the quantification of dependence, numerous measures are available that examine dependence from a wide range of perspectives.
This includes (but is not limited to)
the popular measures of concordance Kendall's tau and Spearman's rho 
(see, e.g., \cite{durante2016principles, nelsen2007introduction}),
the tail dependence coefficients and functions (see, e.g., \cite{joe2014dependence}), or
measures of directed or mutual complete dependence as presented in \cite{chatterjee2020, siburg2013, fgwt2020}.
\\
Apart from the option of assigning a single value to a dependence structure, an alternative way consists in checking whether a copula fulfills certain (positive) dependence properties such as
positive quadrant dependence (PQD),  
left tail decreasingness (LTD), 
right tail increasingness (RTI), 
stochastic increasingness (SI), 
or total positivity of order $2$, the latter usually considered for a copula (TP2) or (if existent) its density (d-TP2 for short), 
where the different notions of positive dependence are linked as given in Figure \ref{fig.Implications.Intro} below
(see, e.g., \cite{joe2014dependence, nelsen2007introduction}).
LTD, RTI and SI are defined coordinatewise while PQD, TP2 and d-TP2 are symmetric in the sense that permuting the coordinates has no effect on the respective dependence property.

Dependence properties may have direct impact on the interplay of different measures of dependence or the equivalence of different modes of convergence:
For instance, as mentioned in \cite{caperaa1993spearman, nefr2007}, Spearman's rho is larger than Kendall's tau for all those copulas that are LTD and RTI.
Moreover, as shown in \cite{siburg2021stochastic} pointwise convergence and weak conditional convergence (see, e.g., \cite{sfx2021weak,sfx2021vine}), i.e., weak convergence of almost all conditional distribution functions, are equivalent for stochastically increasing (SI) copulas.

The property \emph{total positivity of order $2$} has been extensively studied for copulas (TP2) and their densities (d-TP2); see, e.g., \cite{hurli2003, joe1997multivariate, joe2014dependence, lehmann1966}. 
The d-TP2 property is one of the strongest notions of positive dependence - stronger than TP2 and also stronger than stochastic increasingness (SI) (the latter two notions are not related to each other, in general).
Both notions, TP2 and SI, can be examined for any copula, however, the d-TP2 property is defined only for those copulas that are absolutely continuous, which automatically excludes many known copula families such as Fr{\'e}chet copulas or Marshall-Olkin copulas.
\\
In this paper we build the analysis on conditional distributions and investigate total positivity of order $2$ for a copula's Markov kernel (MK-TP2 for short), 
a positive dependence property that is stronger than TP2 and SI, weaker than d-TP2 (see Figure \ref{fig.Implications.Intro} below) but, 
unlike d-TP2, is not restricted to absolutely continuous copulas, making it presumably the strongest dependence property defined for any copula.
From a conditional distribution point of view the MK-TP2 property has been studied in 
\cite{caperaa1990concepts,guillem2000structure} showing its interrelation with the other above-mentioned dependence properties. 
\begin{figure}[h]
\begin{center}
  \begin{tikzcd}[row sep=tiny]
            && \textrm{TP2} \arrow[dr, Rightarrow] &\\
            \textrm{d-TP2} \arrow[r, Rightarrow] & \textrm{MK-TP2} \arrow[ur, Rightarrow] \arrow[dr, Rightarrow] & & \textrm{LTD \& RTI} \arrow[r, Rightarrow] & \textrm{PQD} \\
            && \textrm{SI} \arrow[ur, Rightarrow] &\\
  \end{tikzcd}
\caption{Relations between the different notions of positive dependence.}
\label{fig.Implications.Intro}
\end{center}
\end{figure}
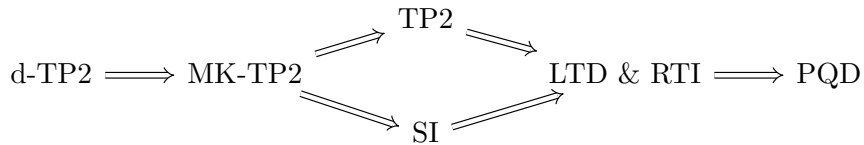

In the present paper, we investigate under which conditions the MK-TP2 property is fulfilled for certain copulas classes, among them the class of Archimedean copulas and the class of extreme value copulas (including Marshall-Olkin copulas). 
To the best of the authors' knowledge, no results are known in this regard so far.
We show that, for Archimedean copulas, the properties MK-TP2 and SI are equivalent, and can be characterized by the Archimedean generator.
In passing we slightly extend a well-known result from literature (see, e.g., \cite{caperaa1993spearman, muller2005}) by proving the equivalence of SI and log-convexity of the generator's derivative for arbitrary Archimedean copulas (i.e., the generator neither needs to be strict nor twice differentiable).
The equivalence of MK-TP2 and SI is consistent with the very restrictive, exchangeable structure of Archimedean copulas and the well-known equivalence of TP2 and LTD (see, e.g., \cite{nelsen2007introduction}).
\\
Another class of copulas generated by convex functions are extreme value copulas (EVC). 
It is well-known that EVCs are TP2 and SI, however, fail to be MK-TP2, in general.
We present sufficient and necessary conditions (in terms of the Pickands dependence function) for an EVC to be MK-TP2 and, 
for certain parametric subclasses (including Marshall-Olkin copulas, Gumbel copulas, and EVCs occuring in the symmetric mixed model by Tawn), we determine exactly those parameters for which their elements are MK-TP2.

The rest of this contribution is organized as follows: 
Section \ref{Sec.Prelim} gathers preliminaries and notation that will be used throughout the paper. 
In Section \ref{Sec.Dep.Prop.}, we formally define the positive dependence properties and examine the MK-TP2 property for different well-known copulas and copula families.
Section \ref{Sec.Arch.} is devoted to Archimedean copulas and comprises the equivalence of MK-TP2 and SI property in the Archimedean setting.
In Section \ref{Sect.EVC} we then present sufficient and necessary conditions for an extreme value copula to be MK-TP2.
Some auxiliary results and technical proofs can be found in Section \ref{Sect.App.}.

\section{Notation and preliminaries} \label{Sec.Prelim}

Throughout this paper we will write $\I := [0,1]$ and let $\mathcal{C}$ denote the family of all bivariate copulas; 
$M$ will denote the comonotonicity copula, 
$\Pi$ the independence copula
and $W$ will denote the countermonotonicity copula.
\\  
For every $C \in \mathcal{C}$ the corresponding probability measure will be denoted by $\mu_C$, 
i.e. $\mu_C([0,u] \times [0,v]) = C(u,v)$ for all $(u,v) \in \I^2$;
for more background on copulas and copula measures we refer to \cite{durante2016principles,nelsen2007introduction}. 
For every metric space $(\Omega,\delta)$ the Borel $\sigma$-field on $\Omega$ will be denoted by $\mathcal{B}(\Omega)$.

In what follows Markov kernels will play a decisive role:
A \emph{Markov kernel} from $\R$ to $\mathcal{B}(\mathbb{R})$ is a mapping 
$K: \mathbb{R}\times\mathcal{B}(\mathbb{R}) \rightarrow \I$ such that for every fixed 
$F\in\mathcal{B}(\mathbb{R})$ the mapping 
$x\mapsto K(x,F)$ is measurable and for every fixed $x\in\mathbb{R}$ the mapping 
$F\mapsto K(x,F)$ is a probability measure. 
Given a real-valued random variable $Y$ and a real-valued random variable $X$ 
on a probability space $(\Omega, \mathcal{A}, P)$ 
we say that a Markov kernel $K$ is a \emph{regular conditional distribution} of $Y$ given $X$ if 
\begin{align*}
	K (X(\omega), F)
	= E ( \mathds{1}_F \circ Y \,|\, X ) (\omega) 
\end{align*}
holds $P$-almost surely for every $F\in \mathcal{B}(\mathbb{R})$. 
It is well-known that for each random vector $(X,Y)$ a regular conditional distribution 
$K(.,.)$ of $Y$ given $X$ always exists 
and is unique for $P^X$-a.e. $x\in\mathbb{R}$,
where $P^X$ denotes the push-forward of $P$ under $X$. 
If $(X,Y)$ has distribution function $H$ 
(in which case we will also write $(X,Y) \sim H$ and let $\mu_H$ denote the corresponding probability measure on $\mathcal{B}(\mathbb{R}^{2})$ we will let 
$K_H$ denote (a version of) the regular conditional distribution of $Y$ given $X$ and simply refer to it as \emph{Markov kernel of $H$}. 
If $C \in \mathcal{C}$ is a copula then we will consider the Markov kernel of $C$ (with respect to the first coordinate) automatically as mapping 
$K_C: \I \times \mathcal{B}(\I) \rightarrow \I$.
Defining the $u$-section of a set $G\in\mathcal{B}(\I^{2})$ as 
$G_u := \lbrace v \in \I \, : \, (u,v) \in G \rbrace$ 
the so-called disintegration theorem yields 
\begin{align}\label{eq:di}
	\mu_C(G)
	= \int_{\I} K_C(u,G_u) \; \mathrm{d} \lambda(u) 
\end{align}
so, in particular, we have 
\begin{align*}
  \mu_C(\I \times F)
	= \int_{\I} K_C(u,F) \; \mathrm{d} \lambda(u)  
	= \lambda(F)
\end{align*}
where $\lambda$ denotes the Lebesgue measure on $\mathcal{B}(\I)$.
For more background on conditional expectation and general disintegration we refer to \cite{Kallenberg, klenke2006wahrscheinlichkeitstheorie};
for more information on Markov kernels in the context of copulas we refer to 
\cite{durante2016principles, sfx2021weak, sfx2021vine}.
\setcounter{section}{2}
\section{Dependence properties} \label{Sec.Dep.Prop.}

In the present section we first summarise some well-known standard notions of positive dependence viewed from a copula perspective;
their probabilistic interpretation is presented in Figure \ref{Fig.Prob.View} below.
We then formally define the MK-TP2 property and discuss this property for different (parametric) copula families.

A copula $C \in \mathcal{C}$ is said to be 
\begin{itemize} \itemsep=0mm
\item
\textit{positively quadrant dependent} (PQD) 
if $C(u,v) \geq \Pi(u,v)$ holds for all $(u,v) \in (0,1)^2$.

\item
\textit{left tail decreasing} (LTD) 
if, for any $v \in (0,1)$, the mapping $(0,1) \to \mathbb{R}$ given by $u \mapsto \frac{C(u,v)}{u}$ is non-increasing.

\item 
\textit{stochastically increasing} (SI) 
if, for (a version of) the Markov kernel $K_C$ and any $v \in (0,1)$, the mapping $u \mapsto K_C(u,[0,v])$ is non-increasing.
\end{itemize}
According to \cite{nelsen2007introduction} these three dependence properties are related as follows
$$ 
  SI 
	\Longrightarrow LTD 
	\Longrightarrow PQD
$$
Another notion of positive dependence that differs from the above mentioned dependence properties is \emph{total positivity of order 2} - a property applicable to various copula-related objects: 
the copula itself, its Markov kernel and its density, 
leading to three different but related positive dependence properties.
In general, a function $f: \Omega \to \R$ with $\Omega \subseteq \R^2$ is said to be \textit{totally positive of order 2} (TP2) on $\Omega$
if the inequality
\begin{align}\label{TP2GeneralEq}
    f(u_1,v_1)f(u_2,v_2) - f(u_1,v_2)f(u_2,v_1) \geq 0 
\end{align}
holds for all $u_1\leq u_2$ and all $v_1 \leq v_2$ such that $[u_1,u_2] \times [v_1,v_2] \subseteq \Omega$. 
The TP2 property has been extensively discussed for copulas and their densities 
(see, e.g., \cite{caperaa1993spearman, joe1997multivariate, joe2014dependence, nelsen2007introduction})
and, to a certain extent, also for their partial derivatives respectively their conditional distribution functions 
(see, e.g., \cite{caperaa1990concepts,guillem2000structure}).
A copula $C \in \C$ is said to be
\begin{itemize} \itemsep=0mm
\item 
\textit{TP2} if the copula itself is TP2, i.e., the inequality
\begin{equation} \label{CTP2.Def}
  C(u_1,v_1) \, C(u_2,v_2) - C(u_1,v_2) \, C(u_2,v_1) \geq 0
\end{equation} 
holds for all $0 < u_1 \leq u_2 < 1$ and all $0 < v_1 \leq v_2 < 1$; 
see Lemma \ref{DP.TP2.Charakterisierung} for a characterization in terms of the copula's Markov kernel.

\item 
\textit{MK-TP2} (short for Markov kernel TP2) if (a version of) its Markov kernel is TP2, i.e., the inequality
\begin{equation} \label{MKTP2.Def}
  K_C(u_1,[0,v_1]) \, K_C(u_2,[0,v_2]) - K_C(u_1,[0,v_2]) \, K_C(u_2,[0,v_1]) \geq 0
\end{equation} 
holds for all $0 < u_1 \leq u_2 < 1$ and all $0 < v_1 \leq v_2 < 1$.
\end{itemize}
If the copula has a density another dependence property can be formulated:
$C$ is said to be
\begin{itemize} 
\item 
\textit{d-TP2} (short for density TP2) if (a version of) its density $c$ is TP2, i.e., the inequality
\begin{equation} \label{dTP2.Def}
  c(u_1,v_1) \, c(u_2,v_2) - c(u_1,v_2) \, c(u_2,v_1) \geq 0
\end{equation} 
holds for all $0 < u_1 \leq u_2 < 1$ and all $0 < v_1 \leq v_2 < 1$. 
\end{itemize}
The property d-TP2 is also referred to as \textit{positively likelihood ratio dependence} (see, e.g., \cite{lehmann1966}). 
Altogether the different notions of positive dependence are linked as depicted in Figure \ref{fig.Implications} 
(compare \cite{guillem2000structure,nelsen2007introduction}). 
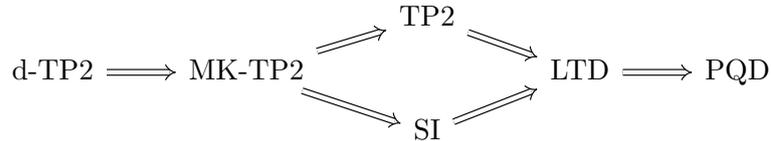
\begin{figure}[h!]
\begin{center}
  \begin{tikzcd}[row sep=tiny]
            && \textrm{TP2} \arrow[dr, Rightarrow] &\\
            \textrm{d-TP2} \arrow[r, Rightarrow] & \textrm{MK-TP2} \arrow[ur, Rightarrow] \arrow[dr, Rightarrow] & & \textrm{LTD} \arrow[r, Rightarrow] & \textrm{PQD} \\
            && \textrm{SI} \arrow[ur, Rightarrow] &\\
  \end{tikzcd}
\caption{Relations between the different notions of positive dependence.}
\label{fig.Implications}
\end{center}
\end{figure}
Notice that the MK-TP2 property, unlike d-TP2, can be examined for any copula, making it presumably the strongest dependence property defined for any copula.

\begin{Remark}{} \label{MKTP2.Def.2}
By definition, 
a copula $C$ that is SI satisfies
$K_C(u_1,[0,v]) \geq K_C(u_2,[0,v])$ 
for all $0 < u_1 \leq u_2 < 1$ and all $v \in (0,1)$. 
Thus, the inequality
\begin{eqnarray*}
  K_C(u_2,[0,v_1])
	& \leq & \min\{K_C(u_1,[0,v_1]),K_C(u_2,[0,v_2])\} \notag
	\\ 
	& \leq & \max\{K_C(u_1,[0,v_1]),K_C(u_2,[0,v_2])\} \notag
	\\
	& \leq & K_C(u_1,[0,v_2]) 
\end{eqnarray*}
holds for all $0 < u_1 \leq u_2 < 1$ and all $0 < v_1 \leq v_2 < 1$.
Therefore, to prove the MK-TP2 property for such copulas Equation (\ref{MKTP2.Def}) only needs to be shown for those rectangles
$0 < u_1 \leq u_2 < 1$ and $0 < v_1 \leq v_2 < 1$ for which $K_C(u_2,[0,v_1]) > 0$.
\end{Remark}{}

We now provide a probabilistic interpretation of the above-mentioned positive dependence properties in terms of continuous random variables $X$ and $Y$ with connecting copula $C$:

\begin{figure}[h]
\begin{center}
    \resizebox{\columnwidth}{!}{%
    \begin{tabular}{|c|l|}
    \hline
    \textbf{Dependence property} & \textbf{Probabilistic interpretation} \\
    \hline &\\
    PQD 
		& $\PP(Y \leq y \vert X\leq x) \geq \PP(Y \leq y)$ \\[3ex]
    \hline &\\
    LTD 
		& $x \mapsto \PP(Y \leq y \vert X\leq x)$ is non-increasing for any $y$  \\[3ex]
    \hline &\\
    SI 
		& $x \mapsto \PP(Y \leq y \vert X = x)$ is non-increasing (a.s.) for any $y$  \\[3ex]
    \hline &\\
    TP2 
		& $x \mapsto \dfrac{\PP(Y \leq y \vert X\leq x)}{\PP(Y \leq y' \vert X\leq x)}$ is non-increasing for any $y\leq y'$  \\[3ex]
    \hline &\\
    MK-TP2 
		& $x \mapsto \dfrac{\PP(Y \leq y \vert X = x)}{\PP(Y \leq y' \vert X = x)}$ is non-increasing (a.s.) for any $y\leq y'$  \\[3ex]
    \hline &\\
    d-TP2 
		& $x \mapsto \dfrac{f_{Y\vert X}(y\vert x)}{f_{Y\vert X}(y'\vert x)}$ is non-increasing (a.s.) for any $y\leq y'$ \\[3ex]
    \hline
    \end{tabular}}
\end{center}
\caption{Probabilistic interpretation of the positive dependence properties}
\label{Fig.Prob.View}
\end{figure}

In the remainder of this section, we examine the MK-TP2 property for different well-known copulas and copula families.

\begin{Example} \leavevmode
\begin{enumerate} \itemsep=0mm
\item 
The independence copula $\Pi$ is d-TP2 and hence MK-TP2.

\item 
The comonotonicity copulas $M$ is MK-TP2; since $M$ is singular, it cannot fulfill the d-TP2 property.

\item
The countermonotonicity copula $W$ fulfills $W(0.5,0.5) < \Pi(0.5,0.5)$.
Thus it fails to be PQD and hence fails to satisfy any of the above-mentioned positive dependence properties.
\end{enumerate}
\end{Example}

The first two copula families under consideration - 
the \textit{Farlie-Gumbel-Morgenstern family of copulas} (FGM copula) 
and the \textit{Gaussian family of copulas} according to \cite{durante2016principles} -
are exceptional since all notions of positive dependence depicted in \eqref{fig.Implications} are equivalent:

\begin{Example} (FGM copula) \\
For $\theta \in [-1,1]$, the mapping $C_\theta: \I^2 \to \I$ given by
\begin{align*}
    C_\theta(u,v) := uv + \theta uv(1-u)(1-v)
\end{align*}
is a copula and called Farlie-Gumbel-Morgenstern (FGM) copula.  
For an FGM copula $C_\theta$ the following statements are equivalent:
\begin{enumerate} \itemsep=0mm
  \item[(a)] $C_\theta$ is PQD / LTD / SI / TP2 / MK-TP2 / d-TP2.
  \item[(b)] $\theta \geq 0$.
\end{enumerate} 
\end{Example}

\begin{Example} (Gaussian copula) \\
For $\rho \in (-1,0) \cup (0,1)$, the mapping $C_\rho: \I^2 \to \I$ given by
$$
    C_\rho(u,v) = \int\limits_{(-\infty,\phi^{-1}(u)] \times (-\infty,\phi^{-1}(v)]} \frac{1}{2\pi \sqrt{1-\rho^2}} \exp{\left(-\frac{s^2-2\rho s t + t^2}{2(1-\rho^2)} \right)} \; \mathrm{d} \lambda^2(s,t)
$$
is a copula and called Gaussian copula, where $\phi^{-1}$ denotes the inverse of the standard Gaussian distribution function. 
According to \cite{joe2014dependence}, for a Gaussian copula $C_\rho$ the following statements are equivalent:
\begin{enumerate} \itemsep=0mm
  \item[(a)] $C_\rho$ is PQD / LTD / SI / TP2 / MK-TP2 / d-TP2.
  \item[(b)] $\rho \in (0,1)$.
\end{enumerate}
\end{Example}

We conclude this section with a brief discussion of the 
\emph{Fr{\'e}chet family of copulas} and
the \emph{Marshall-Olkin family of copulas} according to \cite{durante2016principles} 
- two copula subclasses for which the above-mentioned dependence properties are not equivalent; we thereby provide at the same time an outlook on a result from Section \ref{Sect.EVC}.

\begin{Example} (Fr{\'e}chet copula) \\
For $\alpha, \beta \in \I$ with $\alpha+\beta \leq 1$,
the mapping $C_{\alpha,\beta}: \I^2 \to \I$ given by
$$
  C_{\alpha,\beta}(u,v) = \alpha M(u,v) + (1-\alpha-\beta)\Pi(u,v) + \beta W(u,v)
$$
is a copula and called Fr{\'e}chet copula.
The following statements can be easily verified:
\begin{enumerate} \itemsep=0mm
  \item $C_{\alpha,\beta}$ is PQD / LTD / SI / TP2 if and only if $\beta=0$.
  \item $C_{\alpha,\beta}$ is MK-TP2 if and only if $\alpha \in \{0,1\}$ and $\beta=0$.
	
	\item $C_{\alpha,\beta}$ is d-TP2 if and only if $\alpha=1$ and $\beta=0$.
\end{enumerate}
\end{Example}

\begin{Example} (Marshall-Olkin copula) \label{Intro.Ex.MO} \\
For $\alpha, \beta \in \I$, the mapping $M_{\alpha, \beta}: \I^2 \to \I$ given by
$$
  M_{\alpha, \beta}(u,v) = \begin{cases}
                                u^{1-\alpha} v & u^{\alpha} \geq v^{\beta} \\
                                u v^{1-\beta}  & u^{\alpha} < v^{\beta}
                             \end{cases}
$$
is a copula and called Marshall-Olkin (MO) copula.
Since MO copulas are extreme value copulas (EVC) they are SI and TP2.
If $\min\{\alpha,\beta\} = 0$, then $M_{\alpha, \beta}= \Pi$ implying that in this case $M_{\alpha, \beta}$ is d-TP2.
If, on the other hand, $\min\{\alpha,\beta\} \neq 0$, 
$M_{\alpha, \beta}$ fails to be absolutely continuous (see, e.g., \cite{durante2016principles}) and hence fails to be d-TP2.
In this case, the strongest verifiable notion of positive dependence is MK-TP2:
According to Example \ref{Thm.MO}, for $\alpha, \beta \in (0,1]$, the following statements are equivalent:
\begin{enumerate} \itemsep=0mm
  \item[(a)] $M_{\alpha, \beta}$ is MK-TP2.
  \item[(b)] $\beta=1$.
\end{enumerate}
\end{Example}\pagebreak

In the next two sections we consider the family of Archimedean copulas and the family of extreme value copulas - 
both of which are copula families whose elements are generated by convex functions - 
and investigate under which conditions the corresponding copulas are MK-TP2.

\section{Archimedean copulas} \label{Sec.Arch.}

In this section we investigate under which conditions on the Archimedean generator (or its pseudo-inverse) the corresponding Archimedean copula is MK-TP2.
It is well-known that an Archimedean copula $C$ is TP2 if and only if it is LTD (see, e.g., \cite{nelsen2007introduction}).
In what follows, we add a second characterisation to Figure \ref{fig.Implications} by showing that an Archimedean copula $C$ is MK-TP2 if and only if it is SI.

Recall that a generator of a bivariate Archimedean copula 
is a convex, strictly decreasing function $\varphi: \I \to [0,\infty]$ with $\varphi(1)=0$ (see, e.g., \cite{durante2016principles,nelsen2007introduction}).  
According to \cite{sfx2021weak} we may, w.l.o.g., assume that all generators are right-continuous at $0$.
Every generator $\varphi$ induces a symmetric copula $C$ via
\begin{align*}
    C(u,v) = \psi (\varphi(u) + \varphi(v))
\end{align*}
for all $(u,v) \in \I^2$ where 
$\psi: [0,\infty] \to \I$ denotes the pseudo-inverse of $\varphi$ defined by
$$
  \psi(x)
	:= \begin{cases}
     \varphi^{-1}(x) & \text{if } x\in[0,\varphi(0)) \\
     0              & \text{if } x\geq\varphi(0).
     \end{cases}
$$
The pseudo-inverse $\psi$ is convex, non-increasing and strictly decreasing on $[0,\varphi(0))$. Furthermore, it fulfills $\psi(0)=1$.
Since most of the subsequent results are formulated in terms of $\psi$, we call $\psi$ \emph{co-generator}.
\\
If $\varphi(0) = \infty$ the induced copula $C$ is called strict, 
otherwise it is referred to as non-strict.
Since for every generator $\varphi$ and every constant $a \in (0,\infty)$ the generator $a\varphi$ generates the same copula, 
w.l.o.g., we may assume that generators are normalized, i.e. $\varphi\big(\tfrac{1}{2}\big) = 1$, 
which allows for a one-to-one correspondence between generators and Archimedean copulas.
\\
For every generator $\varphi$ we denote by $D^+\varphi(u)$ ($D^-\varphi(u)$) the right-hand (left-hand) derivative of $\varphi$ at $u \in (0,1)$,
and for every co-generator $\psi$ we denote by $D^+\psi(x)$ ($D^-\psi(x)$) the right-hand (left-hand) derivative of $\psi$ at $x \in (0,\infty)$.
Convexity of $\varphi$ (respectively $\psi$) implies that right-hand and left-hand derivative coincide almost surely, i.e., $\varphi$ (respectively $\psi$) is differentiable outside a countable subset of $(0,1)$ (respectively $(0,\infty)$).
Setting $D^+\varphi(0)=-\infty$ ($D^-\psi(\infty)=0$) in case of strict $C$ allows to view 
$D^+\varphi$ ($D^-\psi$) as non-decreasing and right-continuous (left-continuous) function on $[0,1)$ ($(0,\infty]$).

To investigate under which conditions an Archimedean copula $C$ is MK-TP2 we need to establish (a version of) the Markov kernel of $C$.
To this end, we first define the $0$-level function $f^0: (0,1) \to \I$ by $f^0(u) := \psi(\varphi(0) - \varphi(u))$.

\begin{Theorem}\label{AC.MK}
Let $C$ be an Archimedean copula with generator $\varphi$ and co-generator $\psi$. 
If $C$ is strict, then 
\begin{equation*} 
  K_C(u,[0,v]) = \begin{cases}
                 1 & \text{if } u \in \lbrace 0,1 \rbrace
								 \\
                 \frac{D^-\psi(\varphi(u) + \varphi(v))}{D^-\psi(\varphi(u))} & \text{if } u \in (0,1)
                 \end{cases}
\end{equation*}
is (a version of) its Markov kernel,
and, if $C$ is non-strict, then
\begin{equation*} 
  K_C(u,[0,v]) = \begin{cases}
                 1 & \text{if } u \in \lbrace 0,1 \rbrace
								 \\
                 \frac{D^-\psi(\varphi(u) + \varphi(v))}{D^-\psi(\varphi(u))} & \text{if } u \in (0,1) \text{ and } v \geq f^0(u)
								 \\
                 0 & \text{if } u \in (0,1) \text{ and } v < f^0(u)
                 \end{cases}
\end{equation*}
is (a version of) its Markov kernel.
\end{Theorem}
\begin{proof}
According to \cite{sanchez2015singularity,sfx2021weak}, 
if $C$ is strict, then 
\begin{equation*} 
    K_C(u,[0,v]) 
		= \begin{cases}
      1 & \text{if } u \in \lbrace 0,1 \rbrace
			\\
      \frac{D^+\varphi(u)}{\left( D^+\varphi\right) (C(u,v)) } & \text{if } u \in (0,1) 
      \end{cases}
\end{equation*}
is (a version of) its Markov kernel,
and, if $C$ is non-strict, then
\begin{equation*} 
    K_C(u,[0,v]) = \begin{cases}
                        1 & \text{if } u \in \lbrace 0,1 \rbrace
												\\
                        \frac{D^+\varphi(u)}{\left( D^+\varphi\right) (C(u,v)) } & \text{if } u \in (0,1) \text{ and } v \geq f^0(u) 
												\\
                        0 & \text{if } u \in (0,1) \text{ and } v < f^0(u)
                    \end{cases}
\end{equation*}
is (a version of) its Markov kernel.
Since $\varphi$ is decreasing, for $t \in (0,1)$ if $C$ is strict (and for $t \in [0,1)$ if $C$ is non-strict), $h > 0$ such that $t+h \in (0,1)$, and $l:=\varphi(t+h) - \varphi(t)<0$, we first have
$$
  \frac{\varphi(t+h) - \varphi(t)}{h} 
	= \frac{1}{\frac{\psi(\varphi(t+h)) - \psi(\varphi(t))}{\varphi(t+h) - \varphi(t)}} 
	= \frac{1}{\frac{\psi(\varphi(t)+l) - \psi(\varphi(t))}{l}}
$$
and hence
$$
  D^+\varphi(t) 
  = \lim_{h \to 0^+} \frac{\varphi(t+h) - \varphi(t)}{h} 
	= \lim_{l \to 0^-} \frac{1}{\frac{\psi(\varphi(t)+l) - \psi(\varphi(t))}{l}} 
	= \frac{1}{D^-\psi(\varphi(t))}
$$
Now, consider $u \in (0,1)$. 
If $C$ is strict, 
\begin{align*}
  \frac{D^+\varphi(u)}{\left( D^+\varphi\right) (C(u,v))} 
	&= \begin{cases}
			\frac{D^-\psi(\varphi(u) + \varphi(v))}{D^-\psi(\varphi(u))}
			& \textrm{if } \varphi(u) + \varphi(v) < \varphi(0)
			\\
			0
			& \textrm{if } \varphi(u) + \varphi(v) = \varphi(0)
		\end{cases}
  \\&= \frac{D^-\psi(\varphi(u) + \varphi(v))}{D^-\psi(\varphi(u))}
\end{align*}
If $C$ is non-strict, then $v \geq f^0(u)$ if and only if $\varphi(u) + \varphi(v) \leq \varphi(0)$ and in this case
\begin{align*}
  \frac{D^+\varphi(u)}{\left( D^+\varphi\right) (C(u,v))} 
	&= \begin{cases}
			\frac{D^-\psi(\varphi(u) + \varphi(v))}{D^-\psi(\varphi(u))}
			& \textrm{if } \varphi(u) + \varphi(v) < \varphi(0)
			\\
			\frac{D^-\psi(\varphi(0))}{D^-\psi(\varphi(u))}
			& \textrm{if } \varphi(u) + \varphi(v) = \varphi(0)
		\end{cases}
  \\&= \frac{D^-\psi(\varphi(u) + \varphi(v))}{D^-\psi(\varphi(u))}
\end{align*}
This proves the assertion.
\end{proof}

For $\psi$ twice differentiable, (a version of) the Markov kernel of an Archimedean copula is given in \cite{caperaa1993spearman, muller2005}.

For completeness, we repeat the following necessary and sufficient condition for an Archimedean copula to be TP2 which is well known, see, e.g., \cite{nelsen2007introduction};
having Lemma \ref{DP.TP2.Charakterisierung} in mind, the equivalence of (b) and (c) is straightforward.
A function $f: [0,\infty] \to (0,\infty)$ is called log-convex if $\log f|_{(0,\infty)}$ is convex.

\begin{Proposition}
Let $C$ be an Archimedean copula with generator $\varphi$ and co-generator $\psi$. 
Then the following statements are equivalent:
\begin{enumerate} \itemsep=0mm
  \item[(a)] $C$ is LTD.
  \item[(b)] $C$ is TP2.
  \item[(c)] $\psi$ is log-convex.
\end{enumerate} 
\end{Proposition}

In the sequel we examine under which conditions on the co-generator $\psi$ the corresponding Archimedean copula $C$ is MK-TP2.
We prove the result by showing that SI implies log-convexity of $-D^-\psi$ (Lemma \ref{AC.SI}) which in turn implies MK-TP2 (Lemma \ref{AC.MKTP2}).
In passing we slightly extend a well-known result from literature (see, e.g., \cite{caperaa1993spearman, muller2005}) by proving the equivalence of SI and log-convexity of $-D^-\psi$ for arbitrary Archimedean copulas (i.e., $C$ needs not to be strict and $\psi$ needs not to be twice differentiable).

Recall that every TP2 (and also every SI) copula satisfies $C(u,v)>0$ for all $(u,v) \in (0,1)^2$.
In the Archimedean setting, the latter is equivalent to 
$v > f^0(u)$ for all $(u,v) \in (0,1)^2$.
This implies that a non-strict Archimedean copula is neither TP2 nor SI and hence not MK-TP2.
Therefore, from now on we can restrict ourselves to strict Archimedean copulas.
The following lemma imposes further restrictions on an Archimedean co-generator for the corresponding copula $C$ to be MK-TP2:

\begin{Lemma} \label{AC.Discont.Lemma}
Let $C$ be a strict Archimedean copula with generator $\varphi$ and co-generator $\psi$.  
If $D^-\psi$ has a discontinuity point then $C$ is not SI and hence not MK-TP2.
\end{Lemma}
\begin{proof}
Recall that $C$ is SI if and only if 
$u \mapsto K_C(u, [0, v])$ is non-increasing.
Further recall that $D^-\psi$ is non-decreasing, left-continuous on $(0,\infty]$ and, since $C$ is strict, $D^-\psi(x) < 0$ holds for all $x \in (0,\infty)$.
\\
Since $D^-\psi$ has a discontinuity point there exists $x^\ast\in(0,\infty)$ and $\delta \in (0,1)$ such that 
$y := D^-\psi(x^\ast) < D^-\psi(x^\ast+) = y + \delta < 0$.
Further, there exists $z^\ast \in (0,x^\ast)$ such that $z \mapsto D^-\psi(z)$ is continuous on $(z^\ast,x^\ast)$.
Choosing $\varepsilon \in \big(0, \delta y / (y+\delta) \big)$,
there exists some $w^\ast \in (z^\ast,x^\ast)$ such that 
$ y-\varepsilon \leq D^- \psi(z) \leq y $
for all $z \in [w^\ast,x^\ast]$.
Now, consider 
$u_1 = \psi(x^\ast)$, 
$u_2 = \psi(w^\ast)$,
$v_1 = \psi(1/n)$ for some $n \in \mathbb{N}$ such that $w^\ast+1/n < x^\ast$, and 
$v_2 := 1$.
Then, $u_1 < u_2$, $v_1 < v_2$
and monotonicity implies
\begin{eqnarray*}
  x^\ast + 1/n
	&   =  & \varphi(u_1) + \varphi(v_1)
	\\
	& \geq & \max\{\varphi(u_1) + \varphi(v_2), \varphi(u_2) + \varphi(v_1)\}
	\\
	& \geq & \min\{\varphi(u_1) + \varphi(v_2), \varphi(u_2) + \varphi(v_1)\} 
	\\
	& \geq & \varphi(u_2) + \varphi(v_2)
	\\
	&   =  & w^\ast
\end{eqnarray*}
Thus, using $\varphi(v_2)=0$
\begin{align}\label{Lemma.4.3}
  \frac{K_C(u_2, [0, v_1])} 
	     {K_C(u_1, [0, v_1])}
	& = \frac{D^-\psi(\varphi(u_2) + \varphi(v_1)) \, D^-\psi(\varphi(u_1)+ \varphi(v_2))}
								{D^-\psi(\varphi(u_2)+ \varphi(v_2)) \, D^-\psi(\varphi(u_1) + \varphi(v_1))}
	\\ &  = \frac{D^-\psi(w^\ast + 1/n) \, D^-\psi(x^\ast)}
								{D^-\psi(w^\ast) \, D^-\psi(x^\ast + 1/n)}
	\notag
	 = \frac{D^-\psi(w^\ast + 1/n) \, y}
								{D^-\psi(w^\ast) \, D^-\psi(x^\ast + 1/n)}
	 \notag\\& \geq \frac{y^2}{(y-\varepsilon) (y+\delta)}
	  >    \frac{y^2}{\big(y-\tfrac{\delta y}{y+\delta}\big) (y+\delta)}
	  =    1 \notag
\end{align}
This proves the assertion.
\end{proof}

Equation \eqref{Lemma.4.3} in the proof of Lemma \ref{AC.Discont.Lemma} shows very clearly that in the Archimedean case the SI property is actually a TP2 property.

The next lemma provides a necessary condition for an Archimedean copula to be SI:

\begin{Lemma}\label{AC.SI}
Let $C$ be an Archimedean copula with generator $\varphi$ and co-generator $\psi$. 
If $C$ is SI, then $C$ is strict, $D^-\psi$ is continuous and $-D^-\psi$ is log-convex.
\end{Lemma}\pagebreak
\begin{proof}
Setting $f:= \log(-D^-\psi)$, the copula $C$ is SI if and only if
\begin{eqnarray*}
  0 
	& \leq & \log \left( \frac{K_C(u_1,[0,v])}{K_C(u_2,[0,v])} \right) 
	\\
	&   =  & \log \left( \frac{D^-\psi(\varphi(u_1) + \varphi(v)) \, D^-\psi(\varphi(u_2))}
														{D^-\psi(\varphi(u_1)) \, D^-\psi(\varphi(u_2) + \varphi(v))} 
														\right) 
	\\
	&   =  & f(\varphi(u_1) + \varphi(v)) + f(\varphi(u_2)) - f(\varphi(u_1)) - f(\varphi(u_2) + \varphi(v)) \qquad\qquad (\ast)
\end{eqnarray*}
for all $0 < u_1 \leq u_2 < 1$ and all $v \in (0,1)$ for which $D^-\psi(\varphi(u_1) + \varphi(v)) < 0$.
\\
Assume that $C$ is SI.
Then $C$ is strict (and hence $D^-\psi(x) < 0$ for all $x \in (0,\infty)$) and, by Lemma \ref{AC.Discont.Lemma}, $D^-\psi$ is continuous.
We now prove that $f$ is convex. 
Since $f$ is continuous, it remains to show mid-point convexity.
To this end, consider $x,y \in (0,\infty)$ with $x<y$ and set
$u_1 := \psi((x+y)/2)$,
$u_2:= \psi(x)$, and
$v:=\psi((y-x)/2)$.
Then ($\ast$) implies
\begin{eqnarray*}
  0 
	& \leq & f \big( \tfrac{x+y}{2} + \tfrac{y-x}{2} \big) + f (x) - f \big( \tfrac{x+y}{2} \big) - f \big( x + \tfrac{y-x}{2} \big) 
	\\
	&   =  & f(y) + f(x) - 2 f \big( \tfrac{x+y}{2} \big) 
\end{eqnarray*}
and hence $f \big( (x+y)/2 \big) \leq \big(f(y) + f(x) \big)/2$.
This proves the assertion. 
\end{proof}

The next result provides a sufficient condition for an Archimedean copula to be MK-TP2:

\begin{Lemma}\label{AC.MKTP2}
Let $C$ be an Archimedean copula with generator $\varphi$ and co-generator $\psi$. 
If $-D^-\psi$ is log-convex,
then $C$ is MK-TP2.
\end{Lemma}
\begin{proof}
Defining the functions $f: (-\infty,0) \to \mathbb{R}$ and $G: (0,1)^2 \to (-\infty,0)$ by $f(x):= \log(-D^-\psi (-x))$ and $G(u,v):= -(\varphi(u)+\varphi(v))$, the copula $C$ is MK-TP2 if 
\begin{eqnarray*}
  0 
	& \leq & \log \left( \frac{K_C(u_1,[0,v_1]) \, K_C(u_2,[0,v_2])}{K_C(u_1,[0,v_2]) \, K_C(u_2,[0,v_1])} \right) 
	\\
	&   =  & \log \left( \frac{D^-\psi(-G(u_1,v_1)) \, D^-\psi(-G(u_2,v_2))}
														{D^-\psi(-G(u_1,v_2)) \, D^-\psi(-G(u_2,v_1))} 
														\right) 
	\\
	&   =  & (f \circ G)(u_1,v_1) + (f \circ G)(u_2,v_2) - (f \circ G)(u_1,v_2) - (f \circ G)(u_2,v_1) 
\end{eqnarray*}
for all $0 < u_1 \leq u_2 < 1$ and all $0 < v_1 \leq v_2 < 1$ for which $D^-\psi(-G(u_1,v_1)) < 0$.
Since $G$ is a $2$-increasing and coordinatewise non-decreasing function 
and $f$ is non-decreasing and, by assumption, convex (recall that convexity of $\log(-D^-\psi)$ implies convexity of $f$), 
it follows from \cite[p.219]{marshallolkin2011} 
that the composition $f \circ G$ is $2$-increasing.
This proves the assertion.
\end{proof}

Combining Figure \ref{fig.Implications},
Lemma \ref{AC.MKTP2} and Lemma \ref{AC.SI},
we are now in the position to state the main result of this section:
an Archimedean copula is MK-TP2 if and only if it is SI if and only if $-D^-\psi$ is log-convex.
Recall that we slightly extend a well-known result from literature (see, e.g., \cite{caperaa1993spearman, muller2005}) by proving the equivalence of (b) and (c) for arbitrary Archimedean copulas (i.e., $C$ needs not to be strict and $\psi$ needs not to be twice differentiable).

\begin{Theorem}\label{AC.Eqivalence}
Let $C$ be an Archimedean copula with generator $\varphi$ and pseudo-inverse $\psi$. 
Then the following statements are equivalent:
\begin{enumerate} \itemsep=0mm
  \item[(a)] $C$ is MK-TP2.
  \item[(b)] $C$ is SI.
  \item[(c)] $-D^-\psi$ is log-convex.
\end{enumerate}
\end{Theorem}
\noindent
The previous result is remarkable since for Archimedean copulas the dependence properties SI and MK-TP2 are equivalent.

\begin{Remark}
According to \cite{caperaa1993spearman}, for an Archimedean copula with twice differentiable co-generator $\psi$, 
the following statements are equivalent:
\begin{enumerate} \itemsep=0mm
  \item[(a)] $C$ is d-TP2.
  \item[(b)] $\psi^{''}$ is log-convex.
\end{enumerate} 
\end{Remark}

In sum, for Archimedean copulas, Figure \ref{fig.Implications} reduces to
\begin{equation*}
  \textrm{d-TP2}
	\overset{(1)}{\Longrightarrow}
	\textrm{MK-TP2}
	\Longleftrightarrow
	\textrm{SI}
	\overset{(2)}{\Longrightarrow}
	\textrm{TP2}
	\Longleftrightarrow
	\textrm{LTD}
	\Longrightarrow
	\textrm{PQD}
\end{equation*}
Remark 2.13 in \cite{muller2005} provides an Archimedean copula that is MK-TP2 and SI, but not d-TP2, i.e., the reverse of (1) does not hold, in general.
An Archimedean copula that is LTD and TP2 but not MK-TP2 is given in \cite[Example 19]{spreeuw2013}:
The mapping $\psi: [0,\infty] \to \I$ given by
$\psi(x) := \big(x + \sqrt{1+x^2} \big)^{-1/10}$
is an Archimedean co-generator which is log-convex.
However, since $-\psi'$ is not log-convex, the reverse of (2) does not hold, in general.

An Archimedean co-generator $\psi$ is said to be \emph{completely monotone} if its restiction $\psi|_{(0,\infty)}$ has derivatives of all orders and satisfies 
$(-1)^n \psi^{(n)}(x) \geq 0 $ for all $x \in (0,\infty)$ and all $n \in \N$,
and it follows from \cite{niculescu2006convex} that completely monotone functions are log-convex.
If $\psi$ is completely monotone then its negative derivative is completely monotone as well implying that $(-1)^n \psi^{(n)}$ is log-convex for all $n \in \N$.
Therefore,
every Archimedean copula with completely monotone co-generator is d-TP2 and hence MK-TP2.

\begin{Example} (Gumbel copula) \label{GumbelMKTP2}\\
For $\alpha \in [1,\infty)$, the mapping $\psi_\alpha: [0,\infty] \to \I$ given by 
$$
  \psi_\alpha (x) 
	:= \exp\big(-x^{1/\alpha} \, \log(2) \big) 
$$
is a (normalized) Archimedean co-generator and the corresponding Archimedean copula is called Gumbel copula.
$\psi$ is completely monotone (see, e.g., \cite{nelsen2007introduction}) and hence d-TP2 as well as MK-TP2.
\end{Example}

Recall that Gumbel copulas are Archimedean and extreme value copulas (EVC) at the same time.
EVCs are subject of the next section.

Further examples of Archimedean copulas with completely monotone co-generator may be found in \cite[Chapter 4]{nelsen2007introduction}.

\section{Extreme Value copulas}\label{Sect.EVC}

It is well-known that every extreme value copula (EVC) is TP2 (see, e.g., \cite{joe1997multivariate}) and 
SI (see, e.g., \cite{guillem2000structure}).
However, EVCs may fail to be MK-TP2, in general; see Example \ref{Thm.MO} below or \cite{guillem2000structure}.
In what follows we answer the question under which conditions on the Pickands dependence function the corresponding EVC is MK-TP2.

According to \cite{durante2016principles}, 
a copula $C$ is called \textit{extreme value copula} if one of the following two equivalent conditions is fulfilled 
(see also \cite{deHaan1977,GS2011,Pickands1981}):
\begin{enumerate}
\item There exists a copula $B \in\C$ such that for all $(u,v) \in \I^2$ we have
$$ C(u,v) = \lim_{n \to \infty} B^n(u^{\frac{1}{n}},v ^{\frac{1}{n}}) $$

\item There exists a convex function $A: \I \to \I$ fulfilling 
$\max \lbrace 1-t,t\rbrace \leq A(t) \leq 1$ for all $t \in\I$ such that 
$$C(u,v) = \left(uv \right)^{A\left( \frac{\ln(u)}{\ln(uv)} \right)}$$
for all $(u,v) \in (0,1)^2$.
$A$ is called Pickands dependence function.
\end{enumerate}
Let $\mathcal{A}$ denote the class of all Pickands dependence function. 
Following \cite{sfx2021weak,trutschnig2016mass},
for $A \in \mathcal{A}$ we will let $D^+A$ denote the right-hand derivative of $A$ on $[0,1)$ and 
$D^-A$ the left-hand derivative of $A$ on $(0,1]$. 
Convexity of $A$ implies that it is differentiable outside a countable subset of $(0,1)$, 
i.e. the identity $D^+A(t) = D^-A(t)$ holds for $\lambda$-almost every $t \in (0,1)$. 
Furthermore $D^+A$ is non-decreasing and right-continuous. 
For more information on Pickands dependence functions and the approach via right-hand derivatives
we refer to  \cite{Caperaa1997,Ghoudi1998}.

In order to simplify notation,
for a Pickands dependence function $A$ we define the function $F_A: [0,1) \to \mathbb{R}$ by letting
$$
  F_A(t) := A(t) + (1-t)D^+A(t)
$$ 
and set $F_A(1) := 1$. 
Then, convexity of $A$ implies that $F_A$ is non-decreasing and non-negative (see, e.g., \cite[Lemma 5]{trutschnig2016mass}). 
Additionally, we consider the function $h:(0,1)^2 \to \mathbb{R}$ given by 
$h(u,v) := \log(u)/\log(uv)$ and recall that
$ u \mapsto h(u,v)$ is strictly decreasing for all $v \in (0,1)$, and that 
$ v \mapsto h(u,v)$ is strictly increasing for all $u \in (0,1)$.
Moreover, for $t \in (0,1)$, 
we define the function $f_t: \I \to \I$ by $f_t(u) := u^{\frac{1-t}{t}}$,
and note that $f_t$ is increasing and convex whenever $t \in (0,1/2]$ and concave whenever $t \in [1/2,1) $.
Simple calculation yields $h(u,f_t(u)) = t$.

According to \cite[Lemma 3]{trutschnig2016mass} and using the above-introduced functions,
\begin{equation} \label{Eq.EVC.MK}
    K_A(u,[0,v]) = \begin{cases}
                        1 & \text{if } u \in \lbrace 0,1 \rbrace\\
                        \frac{C(u,v)}{u} F_A(h(u,v)) & \text{if } u,v \in (0,1) \\
                        v & \text{if } (u,v) \in (0,1) \times \lbrace 0,1 \rbrace
                    \end{cases}
\end{equation}
is a (version of the) Markov kernel of the extreme value copula $C$ with Pickands dependence function $A$.

\bigskip
In what follows we give sufficient and necessary conditions on $A$ for the corresponding EVC to be MK-TP2.
As we will show, the behaviour of $D^+A$ at point $0$ is crucial.
The first result is evident:

\begin{Lemma}\label{TheoremTrivialCase}
$D^+A(0) = 0$ if and only if $A \equiv 1$. In this case $C=\Pi$ is d-TP2 and hence MK-TP2.
\end{Lemma}

Next we consider the class of Pickands dependence functions fulfilling $D^+A(0) \in (-1,0)$ 
and show that the corresponding EVCs fail to be MK-TP2.

\begin{Lemma}\label{NecessaryCondition}
If $D^+A(0) \in (-1,0)$ then the corresponding EVC is not MK-TP2.
\end{Lemma}
\begin{proof}
To prove the result we proceed in two steps: 
first, we consider the case when $D^+A$ has at least one discontinuity point and in the second step we assume that $D^+A$ is continuous.

First, suppose there exists some $t_1 \in (0,1)$ such that $D^+A(t_1) \neq D^-A(t_1)$.
By assumption, $D^+A(0) \in (-1,0)$ which implies $F_A(0) = 1 + D^+A(0) \in (0,1]$. 
Since $F_A$ is non-decreasing and continuous outside of a countable subset of $(0,1)$ 
there further exists some $t_2 \in (0,t_1)$ such that $F_A(t) > 0$ for all $t \in [t_2,t_1)$ and $t \mapsto F_A(t)$ is continuous on $[t_2,t_1)$. 
Lemma \ref{LemmaCase2Discontinuous} hence implies that the corresponding EVC is not MK-TP2.

Suppose now that $D^+A$ and thus also $F_A$ are continuous. 
Notice that, since $h$ is continuous, the mapping 
$u \mapsto K_A(u,[0,v])$
is continuous and non-increasing for every $v \in (0,1)$ where the latter property follows from the fact that every EVC is stochastically increasing.\\
Define $\beta_A \in (0,\infty)$ and $g: [0,\beta_A] \to \R$ according to Lemma \ref{LemmaCase2Continuous}. 
Then, by Lemma \ref{LemmaCase2Continuous}, 
$g(\alpha)>g(\beta_A)$,
$F_A\big(1/(1+\beta_A)\big)<F_A\big(1/(1+\alpha)\big)$  
and the inequality 
$$
  \gamma_\alpha
	:= \frac{\log\left(\frac{F_A\big(\tfrac{1}{1+\beta_A}\big)}{F_A\big(\tfrac{1}{1+\alpha}\big)}\right)}{g(\alpha) -g(\beta_A)} < 0
$$
holds for all $\alpha \in(0,\beta_A)$.
Now, fix $\alpha \in(0,\beta_A)$ and choose
$ u := \exp (\gamma_\alpha/2) \in (\exp (\gamma_\alpha),1) $
and set 
$u_1 := u$,
$v_1 := u^{\beta_A}$, 
$v_2 := u^\alpha$ and 
$u_{2,n}:=1-1/n$ for $n \in \mathbb{N}$ with $u < 1-1/n$ (as illustrated in Figure \ref{Fig.EVC.notMK1}). 
\begin{figure}[h]
\begin{center}
  \begin{tikzpicture}[scale=4]
	\draw[->,line width=1] (-0.2,0) -- (1.2,0);
	\draw[->,line width=1] (0,-0.2) -- (0,1.2);
	\draw[-,line width=1] (0,1) -- (1,1);
	\draw[-,line width=1] (1,0) -- (1,1);
	\draw[-,line width=1] (1,0) -- (1,-0.02);
	\node[below=1pt of {(1.05,0)}, scale= 0.75, outer sep=2pt,fill=white] {$1$};
	\draw[-,line width=1] (0,1) -- (-0.02,1);
	\node[left=1pt of {(-0.02,1)}, scale= 0.75, outer sep=2pt,fill=white] {$1$};
	\draw[scale=1, domain=0:1, smooth, variable=\x, blue] plot ({\x}, {\x^0.5});
	\draw[scale=1, domain=0:1, smooth, variable=\x, blue] plot ({\x}, {\x^2});
	\draw[-,line width=1] (0.6,0) -- (0.6,-0.02);
	\node[below=1pt of {(0.6,-0.02)}, scale= 0.75, outer sep=2pt,fill=white] {$u$};
	\draw[-,line width=1] (0.95,0) -- (0.95,-0.02);
	\node[below=1pt of {(0.95,-0.02)}, scale= 0.75, outer sep=2pt,fill=white] {$u_{2,n}$};
	\draw[-,line width=1] (0,0.77) -- (-0.02,0.77);
	\node[left=1pt of {(-0.02,0.77)}, scale= 0.75, outer sep=2pt,fill=white] {$v_2$};
	\draw[-,line width=1] (0,0.36) -- (-0.02,0.36);
	\node[left=1pt of {(-0.02,0.36)}, scale= 0.75, outer sep=2pt,fill=white] {$v_1$};
	\draw[-,line width=1, dashed] (0,0.77) -- (0.95,0.77);
	\draw[-,line width=1, dashed] (0,0.36) -- (0.95,0.36);
	\draw[-,line width=1, dashed] (0.6,0) -- (0.6,0.77);
	\draw[-,line width=1, dashed] (0.95,0) -- (0.95,0.77);
	\node at (0.6,0.36)[circle,fill,inner sep=1.5pt, red]{};
	\node at (0.95,0.36)[circle,fill,inner sep=1.5pt, red]{};
	\node at (0.6,0.77)[circle,fill,inner sep=1.5pt, red]{};
	\node at (0.95,0.77)[circle,fill,inner sep=1.5pt, red]{};
	\node[above=2pt of {(0.55,0.75)}, scale= 0.8, outer sep=2pt,fill=white] {$f_{1/(1+\alpha)}$};
	\node[right=2pt of {(0.6,0.25)}, scale= 0.8, outer sep=2pt,fill=white] {$f_{1/(1+\beta_A)}$};
	\end{tikzpicture}
	\caption{Blue lines are the contour lines of $f_{1/(1+\alpha)}$ and $f_{1/(1+\beta_A)}$ with $\alpha < 1 < \beta_A$.}
  \label{Fig.EVC.notMK1}        
\end{center}				
\end{figure}
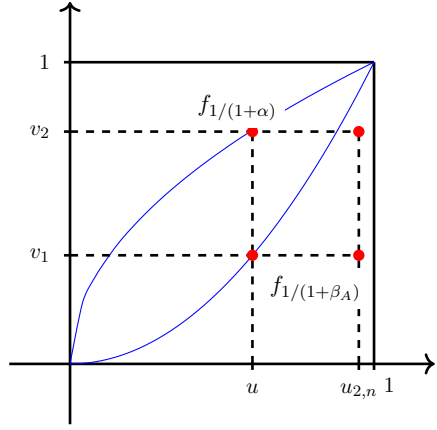
Then
$$
  \lim_{n \to \infty} K_A(u_{2,n},[0,v]) 
	= \lim_{n \to \infty} \frac{C(u_{2,n},v)}{u_{2,n}} F_A(h(u_{2,n},v)) 
	= v F_A(0)
	> 0
$$
for all $v \in (0,1)$ and hence 
$$
  \lim_{n \to \infty} \frac{K_A(u_{2,n},[0,v_2])}{K_A(u_{2,n},[0,v_1])}
	= \frac{u^{\alpha} \, F_A(0)}{u^{\beta_A} \, F_A(0)}
	= u^{\alpha-\beta_A}
	= \exp \big( \tfrac{\gamma_\alpha}{2} (\alpha-\beta_A) \big)
	> 1
$$
Continuity then implies the existence of some $n^\ast \in \mathbb{N}$ with $n^\ast \geq n$ such that
$$
  \frac{K_A(u_{2,n^\ast},[0,v_2])}{K_A(u_{2,n^\ast},[0,v_1])}
	< \underbrace{\left(\frac{F_A\left(\tfrac{1}{1+\alpha}\right)}{F_A\left(\tfrac{1}{1+\beta_A}\right)}\right)^{1/2}}_{>1} u^{\alpha-\beta_A}
$$
Moreover, we have
$$
  K_A(u_{1},[0,v_1])
	= \frac{C(u,u^{\beta_A})}{u} F_A\left(\tfrac{1}{1+\beta_A}\right)
	= u^{(1+\beta_A) A\big(\tfrac{1}{1+\beta_A}\big) - 1} F_A\left(\tfrac{1}{1+\beta_A}\right)
	> 0
$$
and
$$
  K_A(u_1,[0,v_2])
	= \frac{C(u,u^\alpha)}{u} F_A\left(\tfrac{1}{1+\alpha}\right)
	= u^{(1+\alpha)A\left(\frac{1}{1+\alpha}\right) - 1} F_A\left(\tfrac{1}{1+\alpha}\right)
	> 0
$$
Altogether, we finally obtain
\begin{eqnarray*}
  \frac{K_A(u_1,[0,v_1]) \, K_A(u_{2,n^\ast},[0,v_2])}{K_A(u_1,[0,v_2]) \, K_A(u_{2,n^\ast},[0,v_1])}
	& = & \frac{K_A(u_{2,n^\ast},[0,v_2])}{K_A(u_{2,n^\ast},[0,v_1])} \; 
				\frac{u^{(1+\beta_A) A\big(\tfrac{1}{1+\beta_A}\big) - 1} F_A\left(\tfrac{1}{1+\beta_A}\right)}
             {u^{(1+\alpha)A\left(\frac{1}{1+\alpha}\right) - 1} F_A\left(\tfrac{1}{1+\alpha}\right)}
	\\
	& < & \left(\frac{F_A\left(\tfrac{1}{1+\alpha}\right)}{F_A\left(\tfrac{1}{1+\beta_A}\right)}\right)^{1/2} 
	      u^{\alpha-\beta_A} \;
				\frac{u^{g(\beta_A)+\beta_A}  F_A\left(\tfrac{1}{1+\beta_A}\right)}
             {u^{g(\alpha)+\alpha} F_A\left(\tfrac{1}{1+\alpha}\right)}
	\\
	& = & \left(\frac{F_A\left(\tfrac{1}{1+\alpha}\right)}{F_A\left(\tfrac{1}{1+\beta_A}\right)}\right)^{1/2} \; 
							u^{g(\beta_A)-g(\alpha)} \;  
							\frac{F_A\left(\tfrac{1}{1+\beta_A}\right)}{F_A\left(\tfrac{1}{1+\alpha}\right)}
	\\
	& = & \left(\frac{F_A\left(\tfrac{1}{1+\alpha}\right)}{F_A\left(\tfrac{1}{1+\beta_A}\right)}\right)^{1/2} \; 
	      \left(\frac{F_A\left(\tfrac{1}{1+\alpha}\right)}{F_A\left(\tfrac{1}{1+\beta_A}\right)}\right)^{1/2} \; 
				\left(\frac{F_A\left(\tfrac{1}{1+\beta_A}\right)}{F_A\left(\tfrac{1}{1+\alpha}\right)}\right)
  \\
	& = & 1
\end{eqnarray*}
Thus, the corresponding EVC is not MK-TP2. 
This proves the result.
\end{proof}




\begin{Example}\label{RemarkATheta}
In \cite{guillem2000structure} the authors mentioned that the EVC $C_\theta$ occuring in Tawn's model \cite{tawn1988} and generated by the Pickands dependence function 
$A_\theta(t) := \theta t^2 - \theta t + 1$, $\theta \in \I$, 
is MK-TP2 for $\theta = 1$ but fails to be MK-TP2 for $\theta = 1/5$. 
\\
Due to Lemma \ref{TheoremTrivialCase} and Lemma \ref{NecessaryCondition} we are now in the position to finally name exactly those EVCs in that class that are MK-TP2:
Since $D^+(0) = -\theta$, $C_\theta$ is MK-TP2 if and only if $\theta \in \{0,1\}$.
\end{Example}

Due to Lemma \ref{TheoremTrivialCase} and Lemma \ref{NecessaryCondition},
it remains to discuss only those EVCs for which the corresponding Pickands dependence function $A$ satisfies $D^+A(0) = -1$ 
(or, equivalently, $F_A(0)=0$).
We first focus on how many discontinuity points can occur in this setting:

\begin{Lemma} \label{Lemma.EVC.2DP}
Suppose that $D^+A(0) = -1$ and that $D^+A$ has at least two discontinuity points. Then the corresponding EVC is not MK-TP2.
\end{Lemma}
\begin{proof}
Denote by $t_1$ and $t_2$ two discontinuity points of $D^+A$ such that $t \mapsto D^+A(t)$ is continuous on $[t_1,t_2)$. 
Then $t_1$ and $t_2$ are discontinuity points of $F_A$,
$F_A(t) > 0$ for all $t\in[t_1,t_2)$
and the mapping $t \mapsto F_A(t)$ is continuous on $[t_1,t_2)$. 
The assertion now follows from Lemma \ref{LemmaCase2Discontinuous}. 
\end{proof}

The previous lemma states that the Pickands dependence function of an EVC that is MK-TP2 can have at most one discontinuity point.
We can even make the statement more precise for this case.




\begin{Lemma} \label{Lemma.EVC.1DP}
Suppose that $D^+A(0) = -1$ and that $D^+A$ has exactly one discontinuity point at $t^\ast \in (0,1)$. 
If the corresponding EVC is MK-TP2 then $A$ fulfills
$A(t) = 1-t$
for all $t \in [0,t^\ast]$. In this case, $t^\ast \leq 1/2$.
\end{Lemma}
\begin{proof}
Assume that the EVC is MK-TP2.
By Lemma \ref{LemmaCase2Discontinuous}, 
we immediately obtain $F_A(t) = 0$ for all $t \in [0,t^\ast)$ and hence $A(t) = 1-t$ for all $t \in [0,t^\ast]$.
Thus, $t^\ast \leq 1/2$.
This proves the assertion.
\end{proof}

Although quite simple, Lemma \ref{Lemma.EVC.1DP} is crucial, 
as from now on we can restrict ourselves to only those $D^+A$ 
for which there exists some $t^\ast \in [0,1/2]$ such that 
$D^+A(t) = -1$ for all $t \in [0,t^\ast)$ and being continuous on $[t^\ast,1]$.
This is equivalent to considering only those $F_A$ for which there exists some $t^\ast \in [0,1/2]$ such that 
$F_A(t) = 0$ for all $t \in [0,t^\ast)$ and being continuous on $[t^\ast,1]$.
\\
According to inequality (\ref{MKTP2.Def}) and Remark \ref{MKTP2.Def.2}, for proving MK-TP2, it thus remains to consider only those rectangles
$0 < u_1 \leq u_2 < 1$ and $0 < v_1 \leq v_2 < 1$ for which $h(u_2,v_1) \geq t^\ast$ 
(since otherwise $K_A(u_2,[0,v_1])=0$);
setting $f:= \log \circ F_A$
the copula $C$ hence is MK-TP2 if and only if 
\begin{equation}\label{EVC.Ineq.Final}
  1 
	\leq \underbrace{\left( \frac{C(u_1,v_1) \, C(u_2,v_2)}{C(u_1,v_2) \, C(u_2,v_1)}\right)}_{=:I_1}
						\cdot \underbrace{\left( \frac{F_A(h(u_1,v_1)) \, F_A(h(u_2,v_2))}{F_A(h(u_1,v_2)) \, F_A(h(u_2,v_1))} \right)}_{=:I_2}
\end{equation}
for all those rectangles
$0 < u_1 \leq u_2 < 1$ and $0 < v_1 \leq v_2 < 1$ for which $h(u_2,v_1) \geq t^\ast$. 
Since every EVC is TP2, we obvioulsy have $I_1 \geq 1$, and $I_2 \geq 1$ if and only if $F_A \circ h$ is TP2 
(or equivalently, $\log \circ F_A \circ h$ is $2$-increasing) on $h^{-1} \big([t^\ast,1]\big)$.

At this point we summarize our findings obtained so far:

\begin{Theorem} \label{Theorem.EVC.MKTP2}
Let $C$ be an EVC with Pickands dependence function $A$.
\begin{enumerate} \itemsep=0mm
\item 
If $D^+A(0) = 0$, then $C$ is MK-TP2.

\item
If $D^+A(0) \in (-1,0)$, then $C$ is not MK-TP2.

\item
If $D^+A(0) = -1$ and 
\begin{enumerate}\itemsep=0mm
\item[(i)] $D^+A$ has at least two discontinuity points, then $C$ is not MK-TP2.

\item[(ii)] $D^+A$ has exactly one discontinuity point $t^\ast \in (0,1)$ such that $D^+A(t) > -1$ for some $t \in (0,t^\ast)$, then $C$ is not MK-TP2.

\item[(iii)] there exists some $t^\ast \in [0,1/2]$ such that $D^+A(t) = -1$ for all $t \in [0,t^\ast)$, $D^+A$ is continuous on $[t^\ast,1]$ and $F_A \circ h$ is TP2 on $h^{-1} \big([t^\ast,1]\big)$,
then $C$ is MK-TP2.
\end{enumerate}
\end{enumerate}
\end{Theorem}

For the remainder of this section we focus on part 3.(iii) in Theorem \ref{Theorem.EVC.MKTP2}
and examine conditions under which the function $F_A \circ h$ is TP2.
We start with the following lemma that relates the TP2 property of $F_A \circ h$ to log-concavity of $F_A$.

\begin{Lemma} \label{EVC.logconcave} 
Suppose that $D^+A(0) = -1$, there exists some $t^\ast \in [0,1/2]$ such that $D^+A(t) = -1$ for all $t \in [0,t^\ast)$ and $D^+A$ is continuous on $[t^\ast,1]$.
\begin{enumerate} \itemsep=0mm
\item 
If $F_A \circ h$ is TP2 on $h^{-1} \big(\big[t^\ast,1/2\big]\big)$, then $F_A$ is log-concave on $\big[t^\ast,1/2\big]$. 

\item 
If $F_A$ is log-concave on $\big[1/2,1\big)$, 
then $F_A \circ h$ is TP2 on $h^{-1} \big(\big[1/2,1\big)\big)$.
\end{enumerate}
\end{Lemma}
\begin{proof}
We first assume that $F_A\circ h$ is TP2, 
or equivalently $\log \circ F_A \circ h$ is $2$-increasing, on $h^{-1} \big(\big[t^\ast,1/2\big]\big)$. 
Since $F_A$ is assumed to be continuous on $[t^\ast,1/2]$ it is sufficient to show that $\log \circ F_A$ is midpoint concave.
To this end, for every $t^\ast\leq s<t \leq 1/2$ we will construct a rectangle 
$[u_1,v_1] \times [u_2,v_2] \subseteq h^{-1} \big(\big[t^\ast,1/2\big]\big)$ with $h(u_1,v_2) = t, h(u_2,v_1) = s$ and $\max\{h(u_1,v_1),h(u_2,v_2)\} \leq r:=(s+t)/2$
and then use the $2$-increasingness of $\log \circ F_A \circ h$ to show $\log(F_A(s)) + \log(F_A(t)) \leq 2 \log(F_A(r))$ 
(see Figure \ref{LemmaF_ALog_Concave} for an illustration).
\begin{figure}[ht!]
        \begin{center}
            \begin{tikzpicture}[scale=4, domain=0:1]
	\draw[->,line width=1] (-0.2,0) -- (1.2,0);
	\draw[->,line width=1] (0,-0.2) -- (0,1.2);
	\draw[-,line width=1] (0,1) -- (1,1);
	\draw[-,line width=1] (1,0) -- (1,1);
	\draw[-,line width=1] (1,0) -- (1,-0.02);
	\node[below=1pt of {(1,-0.02)}, scale= 0.75, outer sep=2pt,fill=white] {$1$};
	\draw[-,line width=1] (0,1) -- (-0.02,1);
	\node[left=1pt of {(-0.02,1)}, scale= 0.75, outer sep=2pt,fill=white] {$1$};
	\draw[scale = 1, color=blue]   plot[samples=100] (\x,{\x^1.5}); 
	\draw[scale = 1, color=blue]   plot[samples=100] (\x,{\x^3}); 
	\draw[scale = 1, color=blue]   plot[samples=100] (\x,{\x^9}); 
	\draw[-,line width=1] (0.5,0) -- (0.5,-0.02);
	\node[below=1pt of {(0.5,-0.02)}, scale= 0.75, outer sep=2pt] {$u_1$};
	\draw[-,line width=1] (0.75,0) -- (0.75,-0.02);
	\node[below=1pt of {(0.75,-0.02)}, scale= 0.75, outer sep=2pt] {$u_2$};
	\draw[-,line width=1] (0,0.075) -- (-0.02, 0.075);
	\node[left=1pt of {(0,0.075)}, scale= 0.75, outer sep=2pt] {$v_1$};
	\draw[-,line width=1] (0,0.35) -- (-0.02, 0.35);
	\node[left=1pt of {(0,0.35)}, scale= 0.75, outer sep=2pt] {$v_2$};
	\draw[-,line width=1, dashed] (0.5,0) -- (0.5,0.35);
	\draw[-,line width=1, dashed] (0.75,0) -- (0.75,0.35);
	\draw[-,line width=1, dashed] (0,0.075) -- (0.75,0.075);
	\draw[-,line width=1, dashed] (0,0.35) -- (0.75,0.35);
	\node at (0.5,0.075)[circle,fill,inner sep=1.5pt, green]{};
	\node at (0.5,0.35)[circle,fill,inner sep=1.5pt, green]{};
	\node at (0.75,0.075)[circle,fill,inner sep=1.5pt, green]{};
	\node at (0.75,0.35)[circle,fill,inner sep=1.5pt, green]{};
	\end{tikzpicture}
	\caption{Blue lines are the contour lines of $f_{s}, f_{r}$ and $f_{t}$ where $s < r < t < 1/2$.}\label{LemmaF_ALog_Concave}
        \end{center}
\end{figure}
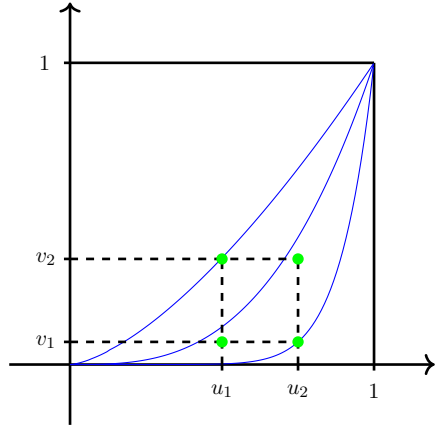
Fix $u_1 \in (0,1)$ and set $v_2 := f_t(u_1)$.
Since $p \mapsto (1-p)/p$ is positive, strictly decreasing and strictly log-convex on $(0,1/2]$ we have
$$\tfrac{1-s}{s} \, \tfrac{1-t}{t} > \left(\tfrac{1-r}{r}\right)^2 \text{ and hence } 1 > \tfrac{1-t}{t} \tfrac{r}{1-r} > \tfrac{1-r}{r} \tfrac{s}{1-s}. $$
We can thus choose $u_2 \in (u_1, 1)$ such that $u_1^{\frac{1-t}{t} \frac{r}{1-r}} < u_2 < u_1^{\frac{1-r}{r} \frac{s}{1-s}}$ and set $v_1 := f_s(u_2)$.
Then, $h(u_1,v_2) = t$, $h(u_2,v_1) = s$ and monotonicity of $h$ yields
\begin{align*}
    h(u_1,v_1) \leq h \big(u_1,u_1^{\frac{1-r}{r}}\big) = r
\end{align*}
and 
\begin{align*}
    h(u_2,v_2) 
		\leq h \left(u_1^{\frac{1-t}{t} \frac{r}{1-r}}, u_1^{\frac{1-t}{t}} \right) 
		  =  \frac{\frac{1-t}{t} \frac{r}{1-r}}{\frac{1-t}{t} \frac{r}{1-r} + \frac{1-t}{t}} =  r
\end{align*}
Applying $2$-increasingness of $\log \circ F_A \circ h$ and monotonicity of $\log \circ F_A$ we finally obtain
\begin{eqnarray*}
  (\log \circ F_A)(s) + (\log \circ F_A)(t)
	&   =  & (\log \circ F_A \circ h)(u_2,v_1) + (\log \circ F_A \circ h)(u_1,v_2)
	\\
	& \leq & (\log \circ F_A \circ h)(u_1,v_1) + (\log \circ F_A \circ h)(u_2,v_2) 
	\\
	& \leq & 2 \log(F_A(r))
\end{eqnarray*}
This proves the first assertion.
The second assertion follows from Lemma \ref{App.Extension.MO}
and the fact that $h$ is $2$-increasing on $h^{-1} \big(\big[1/2,1\big)\big)$.
\end{proof}

\bigskip
For those Pickands dependence functions $A$ that are three times differentiable, 
log-concavity of $F_A$ (i.e., $F_A(t)F_A''(t) - F_A'(t)^2 \leq 0$) is even more visible and Lemma \ref{EVC.logconcave} follows immediately:

\begin{Theorem}\label{EVC.TP2.Aeq}
Suppose that $D^+A(0) = -1$, there exists some $t^\ast \in [0,1/2]$ such that $D^+A(t) = -1$ for all $t \in [0,t^\ast)$ and $A$ is three times differentiable on $(t^\ast,1)$.
Then the following statements are equivalent:
\begin{enumerate} \itemsep=0mm
\item[(a)] $F_A \circ h$ is TP2 on $h^{-1} \big([t^\ast,1)\big)$.

\item[(b)] The inequality 
\begin{equation} \label{EVC.suff.Cond.}
  \frac{F_A(t)F_A''(t) - F_A'(t)^2}{F_A^2(t)} \; t(1-t) + \frac{F_A'(t)}{F_A(t)} \; (1-2t) \leq 0
\end{equation}
holds for all $t \in (t^\ast,1)$.
\end{enumerate}
\end{Theorem}
\begin{proof}
Notice that (a) is equivalent to $\log \circ F_A \circ h$ being 2-increasing on $h^{-1} ([t^\ast,1))$,
i.e., 
\begin{align*}
  0 
	& \leq \frac{\partial^2}{\partial u\partial v} \log(F_A(h(u,v))) \hspace{8cm} (\ast)
	\\
	=  &\frac{F_A(h(u,v))F_A''(h(u,v)) - F_A'(h(u,v))^2}{F_A(h(u,v))^2} \partial_1 h(u,v) \, \partial_2 h(u,v)  \\ & + \frac{F_A'(h(u,v))}{F_A(h(u,v))} \partial_{12} h(u,v)
\end{align*}
for all $(u,v) \in h^{-1} ([t^\ast,1))$, 
where 
\begin{itemize}
\item $\partial_1 h(u,v) := \frac{\partial}{\partial u} h(u,v) = \frac{\log(v)}{u(\log(uv))^2} < 0$
\item $\partial_2 h(u,v) := \frac{\partial}{\partial v} h(u,v) = \frac{-\log(u)}{v(\log(uv))^2} > 0$
\item $\partial_{12} h(u,v) := \frac{\partial^2}{\partial u \partial v} h(u,v)=\frac{\log(u) - \log(v)}{uv (\log(uv))^3}$
\end{itemize}
Since 
\begin{eqnarray*}
  \partial_{12} h(u,v)
	& = & \partial_1 h(u,v) \, \partial_2 h(u,v) \, \frac{1-2h(u,v)}{h(u,v) (1-h(u,v))} 
\end{eqnarray*}
we have \pagebreak
\begin{eqnarray*}
  0 
	& \leq & \frac{\partial^2}{\partial u\partial v} \log(F_A(h(u,v))) 
	\\
	&   =  & \biggl( \frac{F_A(h(u,v))F_A''(h(u,v)) - F_A'(h(u,v))^2}{F_A(h(u,v))^2} + \frac{F_A'(h(u,v))}{F_A(h(u,v))} \frac{1-2h(u,v)}{h(u,v) (1-h(u,v))}\biggr)  \\ 
	&      & \cdot \, \partial_1 h(u,v) \, \partial_2 h(u,v)
\end{eqnarray*}
Therefore, $(\ast)$ holds if and only if 
$$
  0 \geq \frac{F_A(t)F_A''(t) - F_A'(t)^2}{F_A^2(t)} \; t(1-t) + \frac{F_A'(t)}{F_A(t)} \; (1-2t)
$$
for all $t \in (t^\ast,1)$. This proves the assertion.
\end{proof}

\begin{Remark}
It is worth mentioning that Inequality (\ref{EVC.suff.Cond.}) is equivalent to the non-increasingness of 
\begin{equation}\label{EVC.suff.Cond.2}
  t \mapsto t\, (1-t) \, \frac{F_A'(t)}{F_A(t)}
\end{equation}
on $(t^\ast,1)$ which turns out to be neither a sufficient nor a necessary condition for log-concavity of $F_A$ 
(i.e., $ t \mapsto F_A'(t)/F_A(t) $ is non-increasing) on $(t^\ast,1)$.
\end{Remark}

Lemma \ref{EVC.logconcave} and Theorem \ref{EVC.TP2.Aeq} suggest that log-concavity of $F_A$ could be a sufficient condition for MK-TP2. The next example contradicts this conjecture.

\begin{Example} \label{Ex.Pickands.Log}
The mapping $A: \I \to \I$ given by
$$
  A(t) 
	:= \log\left( t^{3/2} + (1-t)^{3/2} \right)^{2/3} + 1
$$
is a Pickands dependence function fulfilling $D^+A(0) = -1$ and being (at least) three times differentiable on $(0,1)$ with
\begin{eqnarray*}
  F_A(t) 
	& = & \frac{2}{3}\log\left( t^{3/2} + (1-t)^{3/2} \right) + \frac{t^{1/2}}{t^{3/2} + (1-t)^{3/2}}
	\\
	F_A'(t) 
	& = & \frac{1-t}{(t^{3/2} + (1-t)^{3/2})^2} \, \frac{1 + 4t(1-t) - 2\sqrt{(1-t)t}}{2\sqrt{t(1-t)}}
\end{eqnarray*}
Figure \ref{Fig.Pickands.Log} depicts a sample of the corresponding EVC.
\begin{figure}[h!]
		\centering
		\includegraphics[width=0.7\textwidth]{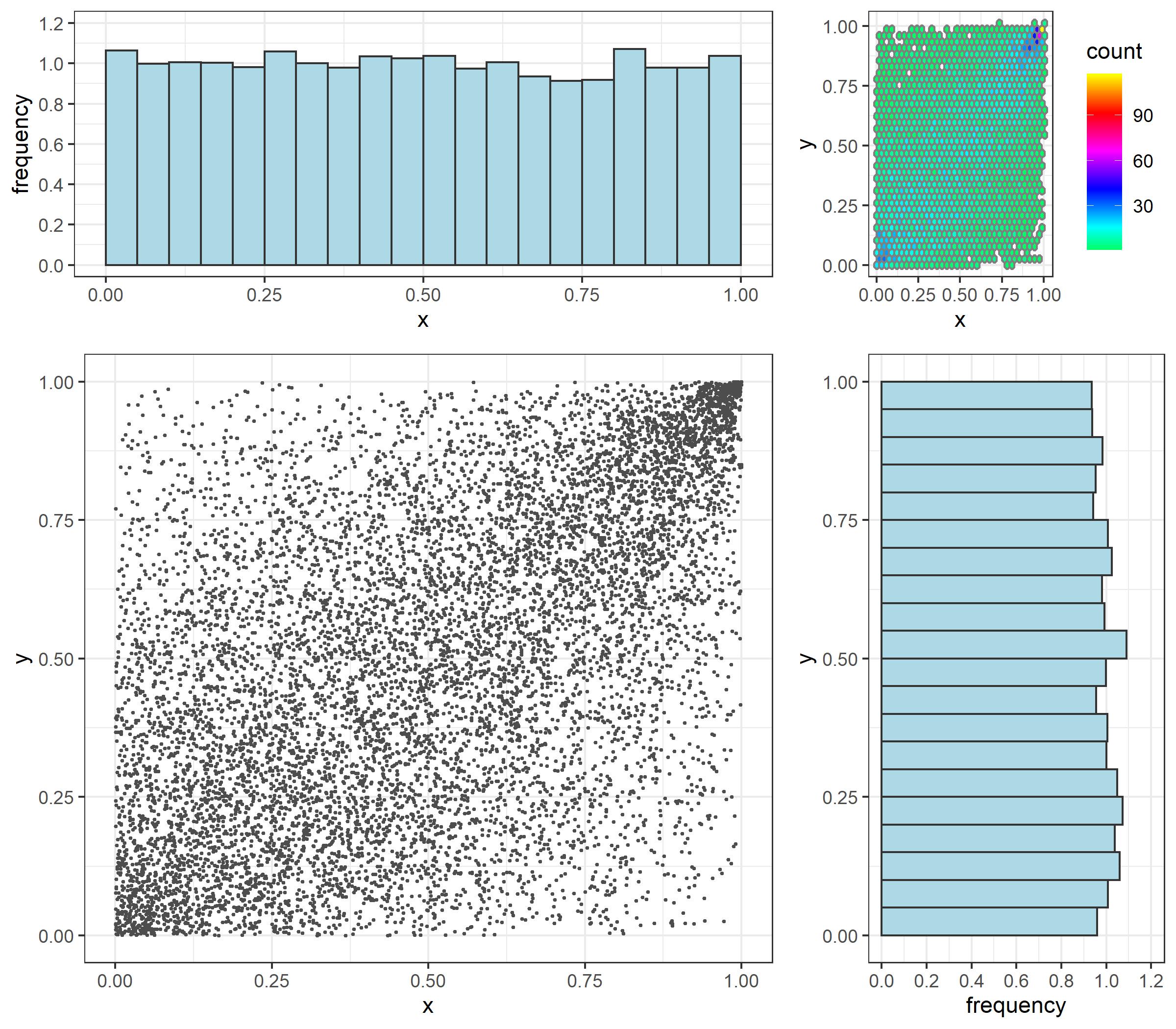}
		\caption{Sample of size $n=10.000$ drawn from the EVC discussed in Example \ref{Ex.Pickands.Log} with Pickands dependence function $ A(t) 
	:= \log\left( t^{3/2} + (1-t)^{3/2} \right)^{2/3} + 1$.}
		\label{Fig.Pickands.Log}
\end{figure}
Tedious calculation yields that $F_A$ is log-concave.
However, for $t_1 = 0.1$ and $t_2 = 0.2$
$$
  \frac{F_A'(t_1)}{F_A(t_1)}t_1(1-t_1) = 0.474 < 0.505 = \frac{F_A'(t_2)}{F_A(t_2)}t_2(1-t_2)
$$
which contradicts (\ref{EVC.suff.Cond.2}), hence $F_A \circ h$ fails to be TP2 on $(0,1)^2$. 
It even holds that the corresponding EVC is not MK-TP2 which follows considering the rectangle 
$[0.9,0.95] \times [0.5,0.6]$.
\end{Example}

In what follows, we examine for selected classes of EVCs for which parameters the MK-TP2 property is fulfilled and for which it is not.

\begin{Example} (Gumbel copula) \\
For $\alpha \in [1,\infty)$, the mapping $A_\alpha: \I \to \I$ given by 
$$
  A_\alpha (t) 
	:= (t^\alpha + (1-t)^\alpha)^{1/\alpha}
$$
is a Pickands dependence function; the corresponding EVC $C_\alpha$ is called Gumbel copula (see, e.g., \cite{durante2016principles}).
It immedately follows from Example \ref{GumbelMKTP2} that $C_\alpha$ is MK-TP2 by taking into account its Archimedean structure.
\end{Example}

We now consider the asymmetric mixed model by Tawn (see, e.g., \cite{gene2009, tawn1988}) which generalizes Example \ref{RemarkATheta}:

\begin{Example} (Asymmetric mixed model by Tawn) \\
For $\theta$ and $\kappa$ satisfying $\theta \geq 0$, $\theta+3\kappa\geq0$, $\theta+\kappa\leq1$ and $\theta + 2\kappa \leq 1$, 
the mapping $A: \I \to \I$ given by
$$
  A(t) := 1 - (\theta+\kappa)t+ \theta t^2 + \kappa t^3
$$
is a Pickands dependence function; the corresponding EVC $C_{\theta,\kappa}$ occurs in Tawn's model \cite{tawn1988}.
Note that the parameter conditions imply $0 \leq \theta+\kappa \leq 1$. 
The following statements are equivalent
\begin{enumerate} \itemsep=0mm
\item[(a)] $C_{\theta,\kappa}$ is MK-TP2.
\item[(b)] $\theta+\kappa \in \{0,1\}$.
\end{enumerate}
Indeed, considering $A'(t) = - (\theta+\kappa)+ 2\theta t + 3\kappa t^2$ and hence $A'(0)= - (\theta+\kappa)$ 
cases $\theta+\kappa = 0$ and $\theta+\kappa \in (0,1)$ immediately follow from Theorem \ref{Theorem.EVC.MKTP2}.
Now assume that $\theta+\kappa = 1$. 
Then $A(t) = 1 - t+ \theta t^2 + (1-\theta) t^3$ with $\theta \in [1,3/2]$.
Simple calculation yields
\begin{itemize}
\item $F_A(t)  = 2 \theta t(1-t)^2 + (3-2t)t^2$
\item $F_A'(t) = 2(1-t)(\theta+3(1-\theta)t)$
\item $F_A''(t) = -2 \theta + 6 (1 - \theta) (1 - t) - 6 (1 - \theta) t$
\end{itemize}
In order to show inequality (\ref{EVC.suff.Cond.}) we calculate
\begin{eqnarray*} 
  \lefteqn{\big( F_A(t)F_A''(t) - F_A'(t)^2\big) \; t(1-t) + F_A(t) \, F_A'(t) \; (1-2t)}
	\\
	& = & 2 (1-t) t^2 \, \big( - 6 \theta^2 (1-t)^3 - 6 (2-t) t^2 +  \theta (3 - 17 t + 30 t^2 - 12 t^3) \big)
\end{eqnarray*}\pagebreak
The latter term is negative if and only if 
$$
    \theta \underbrace{(3 - 17 t + 30 t^2 - 12 t^3)}_{>0}
		\geq 6 \theta^2 (1-t)^3 + 6 (2-t) t^2
$$
for all $t \in (0,1)$. 
Bounding $\theta$ on the left-hand side from below by $1$ and on the right-hand side from above by $3/2$ yields a sufficient condition, namely
$$
    (3 - 17 t + 30 t^2 - 12 t^3)
		\geq \tfrac{27}{2} (1-t)^3 + 6 (2-t) t^2
$$
which is true for all $t \in (0,1)$.
Therefore, inequality (\ref{EVC.suff.Cond.}) holds implying that $C_{\theta,\kappa}$ is MK-TP2.
\end{Example}

We now resume Example \ref{Intro.Ex.MO} and show that a Marshall-Olkin copula $M_{\alpha, \beta}$ with $\alpha, \beta \in (0,1]$ is MK-TP2 if and only if $\beta=1$:

\begin{Example} (Marshall-Olkin copula) \label{Thm.MO} \\
For $\alpha, \beta \in (0,1]$, the mapping $A: \I \to \I$ given by
$$
  A(t) = \begin{cases}
           1-\beta t 
				   & t < \frac{\alpha}{\alpha + \beta} 
					 \\
					 1-\alpha(1-t)
					 & t \geq \frac{\alpha}{\alpha + \beta}
         \end{cases}
$$
is a Pickands dependence function; the corresponding EVC $M_{\alpha, \beta}$ is called Marshall-Olkin copula (see, e.g., \cite{durante2016principles}).
The following statements are equivalent
\begin{enumerate} \itemsep=0mm
\item[(a)] $M_{\alpha, \beta}$ is MK-TP2.
\item[(b)] $\beta = 1$.
\end{enumerate}
Indeed, we first have $D^+A(0) = -\beta$.
Now, if $\beta \in (0,1)$ then $D^+A(0) \in (-1,0)$ and Theorem \ref{Theorem.EVC.MKTP2} implies that $M_{\alpha, \beta}$ is not MK-TP2. 
If, otherwise, $\beta = 1$ then  
$$
  F_A(t) 
	= \begin{cases}
      0 
			& t < \frac{\alpha}{\alpha + 1} 
			\\
      1
			& t \geq \frac{\alpha}{\alpha + 1}
    \end{cases} 
  = \mathds{1}_{[\frac{\alpha}{\alpha + 1},1]}(t)                             
$$
implying Inequality \eqref{EVC.suff.Cond.} to hold for all $t \in \big(\alpha/(\alpha + 1),1\big)$.
\end{Example}

We conclude this section by drawing the reader's attention to another quite simple condition under which a given EVC fails to be MK-TP2; the proof is deferred to the Appendix.

\begin{Lemma}\label{LemmaFAConstantNotMKTP2}
Suppose that $D^+A(0) = -1$ and that there exist $t_1, t_2 \in (0,1)$ with $t_1 < t_2$ and some $c\in(0,1)$ 
such that $F_A(t) = c$ on $[t_1, t_2]$. 
Then the corresponding EVC is not MK-TP2.
\end{Lemma}

In other words, if the function $F_A$ is constant on an interval, then the corresponding EVC is not MK-TP2.
We illustrate this result with an example:

\begin{Example}{}\label{CounterExampleJump}
The mapping $A: \I \to \I$ given by
$$
  A(t) 
	:= \begin{cases}
	    1-t 
			& t \in \big[0,\tfrac{1}{8}\big)
			\\
			\tfrac{15}{16}-\tfrac{1}{2} t 
			& t \in \big[\tfrac{1}{8},\tfrac{1}{4}\big)
			\\
			\tfrac{12}{16}+\big(x-\tfrac{1}{2}\big)^2 & t \in \big[\tfrac{1}{4},1\big]
		 \end{cases}
$$
is a Pickands dependence function fulfilling $D^+A(t) = -1$ for all $t \in \big[0,1/8\big)$ and 
$$
  F_A(t) 
	= \begin{cases}
	    0 
			& t \in \big[0,\tfrac{1}{8}\big)
			\\
			\tfrac{7}{16}
			& t \in \big[\tfrac{1}{8},\tfrac{1}{4}\big)
			\\
			t(2-t) & t \in \big[\tfrac{1}{4},1\big]
		 \end{cases}
$$
(see Figure \ref{FigureCounterExampleJump} for an illustration).
It is straightforward to verify that $F_A \circ h$ fails to be TP2 on $\big[1/8,1\big]$ and it follows from Lemma \ref{LemmaFAConstantNotMKTP2} that the corresponding EVC is not MK-TP2. 
\begin{figure}[ht!]
    \centering
    \begin{subfigure}{.45\textwidth}
        \centering
        \begin{tikzpicture}[scale=4, domain=0:1]
	\draw[->,line width=1] (-0.2,0) -- (1.2,0);
	\draw[->,line width=1] (0,-0.2) -- (0,1.2);
	\draw[-,line width=1] (1,0) -- (1,-0.02);
	\node[below=1pt of {(1,-0.02)}, scale= 0.75, outer sep=2pt,fill=white] {$1$};
	\draw[-,line width=1] (0,1) -- (-0.02,1);
	\node[left=1pt of {(-0.02,1)}, scale= 0.75, outer sep=2pt,fill=white] {$1$};
	\draw[-,line width=1, dashed] (0,0) -- (1,1); 
	\draw[-,line width=1, dashed] (0,1) -- (1,0); 
	\draw[scale = 1, color=blue]   plot[samples=100, domain=0:0.125] (\x,{1-\x});
	\draw[scale = 1, color=blue]   plot[samples=100, domain=0.125:0.25] (\x,{15/16-0.5*\x});
	\draw[scale = 1, color=blue]   plot[samples=100, domain=0.25:1] (\x,{12/16+(\x-0.5)^2});
	\end{tikzpicture}
	    \caption{Plot of Pickands dependence function $A$.}
    \end{subfigure}%
    \begin{subfigure}{.45\textwidth}
        \centering
        \begin{tikzpicture}[scale=4, domain=0:1]
	\draw[->,line width=1] (-0.2,0) -- (1.2,0);
	\draw[->,line width=1] (0,-0.2) -- (0,1.2);
	\draw[-,line width=1] (1,0) -- (1,-0.02);
	\node[below=1pt of {(1,-0.02)}, scale= 0.75, outer sep=2pt,fill=white] {$1$};
	\draw[-,line width=1] (0,1) -- (-0.02,1);
	\node[left=1pt of {(-0.02,1)}, scale= 0.75, outer sep=2pt,fill=white] {$1$};
	\draw[scale = 1, color=blue]   plot[samples=100, domain=0:0.125] (\x,{0});
	\draw[scale = 1, color=blue]   plot[samples=100, domain=0.125:0.25] (\x,{7/16});
	\draw[scale = 1, color=blue]   plot[samples=100, domain=0.25:1] (\x,{\x*(2-\x)});
	\node at (0.125,7/16)[circle,fill,inner sep=1.5pt, blue]{};
	\draw[blue] (0.125,0) circle (0.015);
	\end{tikzpicture}
	    \caption{Plot of function $F_A$.}
    \end{subfigure}\caption{Corresponding functions $A$ and $F_A$ for Example \ref{CounterExampleJump}.}
		\label{FigureCounterExampleJump}
\end{figure}
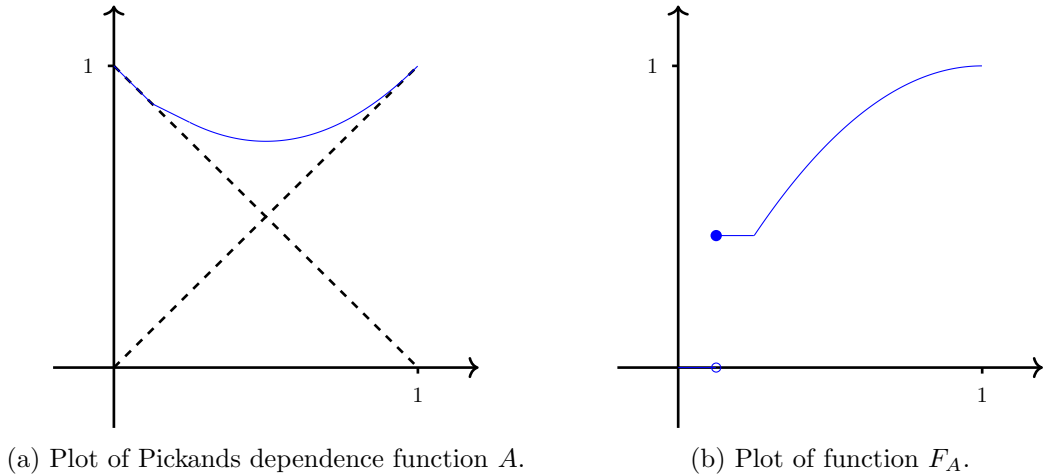
\end{Example}{}

\section*{Acknowledgement}
The first author gratefully acknowledges the support of the WISS 2025 project 
'IDA-lab Salzburg' (20204-WISS/225/197-2019 and 20102-F1901166-KZP).
The second author gratefully acknowledges the financial support from AMAG Austria Metall AG within the project ProSa.









\section{Supplementary material}\label{Sect.App.}

In this appendix we gather some helpful results concerning the TP2 property and extreme value copulas.

By definition, 
a copula $C$ that is TP2 (or SI) satisfies
$C(u,v) \geq \Pi(u,v) > 0$ for all $(u,v) \in (0,1)^2$.
The next result characterizes the TP2 property in terms of a monotonicity property:

\begin{Lemma}{} \label{DP.TP2.Charakterisierung}
Consider some copula $C$ satisfying $C(u,v) > 0$ on $(0,1)^2$.
Then the following statements are equivalent:
\begin{enumerate} \itemsep=0mm
\item[(a)] $C$ is TP2.
\item[(b)] $\log \circ C$ is $2$-increasing.
\item[(c)] For any $u \in (0,1)$, the mapping $(0,1) \to \mathbb{R}$ given by
$$
  v \mapsto \frac{K_C(u,[0,v])}{C(u,v)} 
$$
is non-decreasing.
\end{enumerate}
\end{Lemma}{}
\begin{proof}{}
The equivalence of (a) and (b) is straightforward.
Now, fix $v \in (0,1)$ and define $F_v: \I \to \mathbb{R}$ by letting $F_v(u) := C(u,v)$.
Then $F_v$ is a continuous distribution function with $\lambda$-density satisfying
$$
  F_v(u) 
	= \int_{(0,u]} K_C(s,[0,v]) \; \mathrm{d} \lambda(s)
$$
Thus, for every $u \in (0,1)$
\begin{eqnarray*}
  (\log \circ C)(u,v)
	& = & - \, \int_{(C(u,v),1]} \frac{1}{s} \; \mathrm{d} \lambda(s)
	\\ & = & \;\; - \, \int_{(0,1)} \mathds{1}_{(u,1]}(F_v^\leftarrow(s)) \frac{1}{(F_v \circ F_v^\leftarrow)(s)} \; \mathrm{d} \lambda(s)
	\\
	& = & - \, \int_{(u,1]} \frac{1}{C(s,v)} \; \mathrm{d} \lambda^{F_v^\leftarrow}(s)
	\;\; = \;\; - \, \int_{(u,1]} \frac{K_C(s,[0,v])}{C(s,v)} \; \mathrm{d} \lambda(s)
\end{eqnarray*}
Therefore,
$\log \circ C$ is $2$-increasing if and only if, for any $u \in (0,1)$ 
$$
  v \mapsto \frac{K_C(u,[0,v])}{C(u,v)} 
$$
is non-decreasing. This proves the result.
\end{proof}{}

It is well-known that if $\phi$ is a $2$-increasing and monotone (both coordinates in the same direction) function and $f$ is convex and non-decreasing, 
then the composition $f \circ \phi$ is $2$-increasing (see \cite[p.219]{marshallolkin2011}).
In the next lemma we modify the assumption on $f$ for the case $\phi$ is monotone in opposite direction.
Lemma \ref{App.Extension.MO} applies to Lemma \ref{EVC.logconcave}.

\begin{Lemma} \label{App.Extension.MO}
Suppose that $\phi: \Omega_2 \to \Omega_1$, $\Omega_2 \subseteq \R^2$, $\Omega_1 \subseteq \R$, is $2$-increasing, 
$x \mapsto \phi(x,y)$ is non-increasing, 
$y \mapsto \phi(x,y)$ is non-decreasing,
and $f:\Omega_1 \to \R$ is concave and non-decreasing.
Then $f \circ \phi$ is $2$-increasing. 
\end{Lemma}

\begin{proof}
For $x_1 \leq x_2$ and $y_1\leq y_2$ such that $[x_1,x_2] \times [y_1,y_2] \subseteq \Omega_2$ monotonicity of $\phi$ implies
\begin{align*}
    \phi(x_2,y_1) \leq \min\{ \phi(x_1,y_1), \phi(x_2,y_2) \} \leq \max\{ \phi(x_1,y_1), \phi(x_2,y_2) \} \leq \phi(x_1,y_2).
\end{align*}
and $2$-increasingness yields 
$\phi(x_1,y_2) + \phi(x_2,y_1) \leq \phi(x_1,y_1) + \phi(x_2,y_2) $.
\\
First, assume $\phi(x_1,y_1) < \phi(x_2,y_2)$, choose $s_2$ and $t_1 < t_2$ such that $s_2-t_1=\phi(x_1,y_2)$ and $s_2-t_2 = \phi(x_2,y_2)$, and set $s_1:=\phi(x_2,y_1) + t_2 < s_2$. Then (see Figure \ref{FigureFALogConcanve} for an illustration)
\begin{align*}
    s_1 - t_1 
		  &= \phi(x_2, y_1) + t_2 - t_1 + \phi(x_1,y_2) - s_2 + t_1 
    \\&= \phi(x_2, y_1) + \phi(x_1,y_2) + t_2  - s_2
    \\&\leq \phi(x_1, y_1) + \phi(x_2,y_2) + t_2  - s_2
    \\&= \phi(x_1, y_1)
\end{align*}
\begin{figure}[ht!]
        \begin{center}
            \begin{tikzpicture}[scale=4, domain=0:1]
	\draw[-,line width=1] (-1.2,0) -- (1.2,0);	
	\draw[-,line width=1] (1,0) -- (1,-0.02);
	\node[below=1pt of {(1,-0.02)}, scale= 0.75, outer sep=2pt,fill=white] {$\phi(x_1,y_2)$};	
	\draw[-,line width=1] (-1,0) -- (-1,-0.02);
	\node[below=1pt of {(-1,-0.02)}, scale= 0.75, outer sep=2pt,fill=white] {$\phi(x_2,y_1)$};
	\draw[-,line width=1] (-0.25,0) -- (-0.25,-0.02);
	\node[below=1pt of {(-0.25,-0.02)}, scale= 0.75, outer sep=2pt,fill=white] {$\phi(x_1,y_1)$};	
	\draw[-,line width=1] (0.5,0) -- (0.5,-0.02);
	\node[below=1pt of {(0.5,-0.02)}, scale= 0.75, outer sep=2pt,fill=white] {$\phi(x_2,y_2)$};	
	\draw[-,line width=1, dashed, color = red] (1,0.18) -- (1,0); 
	\node[above=1pt of {(1,0.18)}, scale= 0.75, outer sep=2pt,fill=white] {\textcolor{red}{$s_2-t_1$}};	
	\draw [->, color=red] (1,0) to [out=90,in=90] (0.5,0);	
	\draw[-,line width=1, dashed, color = red] (0.5,0.18) -- (0.5,0); 
	\node[above=1pt of {(0.5,0.18)}, scale= 0.75, outer sep=2pt,fill=white] {\textcolor{red}{$s_2-t_2$}};	
	\draw [decorate,decoration = {brace,mirror,amplitude=10pt}] (0.5,0) --  (1,0);
	\draw [decorate,decoration = {brace,mirror,amplitude=10pt}] (-1,0) --  (-0.5,0);
	\draw [->, color=blue] (0.75,-0.1) to [out=-90,in=-90] (-0.75,-0.1);	
	\draw[-,line width=1, dashed, color = red] (-0.5,0.18) -- (-0.5,0); 
	\node[above=1pt of {(-0.5,0.18)}, scale= 0.75, outer sep=2pt,fill=white] {\textcolor{red}{$s_1-t_1$}};
	\draw[-,line width=1, dashed, color = red] (-1,0.18) -- (-1,0); 
	\node[above=1pt of {(-1,0.18)}, scale= 0.75, outer sep=2pt,fill=white] {\textcolor{red}{$s_1-t_2$}};
	\end{tikzpicture}
	\caption{Illustration how the specific points are chosen and why $s_1 - t_1 \leq \phi(x_1,y_1)$ holds in the case $\phi(x_1,y_1) < \phi(x_2,y_2)$.} \label{FigureFALogConcanve}
        \end{center}
\end{figure}
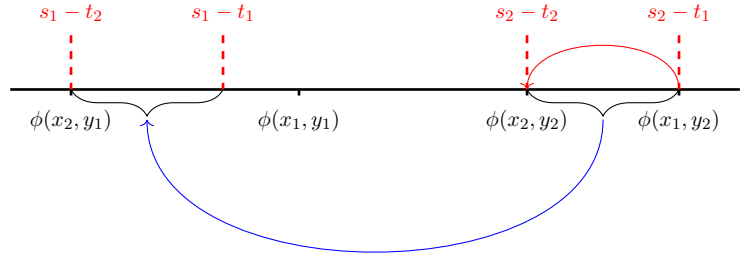
By assumption $\exp\circ f$ is (non-negative), log-concave and non-decreasing. 
Analogous to \cite{schoenberg1951} (see also \cite[Example A.10]{marshallolkin2011}) log-concavity implies that the function $(s,t) \mapsto (\exp \circ f) (s-t)$ is TP2 which gives
\begin{align*}
    (\exp \circ f)(\phi(x_1, y_2)) \cdot (\exp \circ f)(\phi(x_2, y_1)) 
		  &  =  (\exp \circ f)(s_2-t_1) \cdot (\exp \circ f)(s_1-t_2) 
		\\&\leq (\exp \circ f)(s_2-t_2) \cdot (\exp \circ f)(s_1-t_1) 
		\\&  =  (\exp \circ f)(\phi(x_2, y_2)) \cdot (\exp \circ f)(s_1-t_1)
\end{align*}
and monotonicity of $\exp \circ f$ yields
\begin{align*}
    (\exp \circ f)(\phi(x_1, y_2)) \cdot (\exp \circ f)(\phi(x_2, y_1))  
		\leq (\exp \circ f)(\phi(x_2, y_2)) \cdot (\exp \circ f)(\phi(x_1, y_1))
\end{align*}
This is equivalent to $f \circ \phi$ being $2$-increasing. The case $\phi(x_1, y_1) > \phi(x_2, y_2)$ follows analogously. This proves the desired assertion.
\end{proof}

We now present some auxiliary results that are used in Section \ref{Sect.EVC}.
Lemma \ref{LemmaCase2Discontinuous} provides a structure for $F_A$ that contradicts the MK-TP2 property.

\begin{Lemma}\label{LemmaCase2Discontinuous}
Suppose that $F_A$ has discontinuity point $t_r \in (0,1)$ and there exists some $t_l\in(0,t_r)$ 
such that $F_A(t) > 0$ for all $t\in[t_l,t_r)$ and $t \mapsto F_A(t)$ is continuous on $[t_l,t_r)$. 
Then the corresponding EVC is not MK-TP2.
\end{Lemma}
\begin{proof}
In what follows we construct a rectangle $[u_1,u_2] \times [v_1,v_2]$ for which inequality \eqref{MKTP2.Def} is not fulfilled (see Figure \ref{FigureAppendix1} for an illustration). 
\\
Recall that $F_A$ is non-decreasing.
First of all, discontinuity of $F_A$ in $t_r$ implies the existence of some $\delta \in (0,1)$ such that 
$F_A(t_r) = \delta + y$ where $y:= \sup_{t<t_r} F_A(t) > 0$, and choose $\varepsilon \in \bigl(0, \delta y/ (\delta + y) \bigr)$.
Since $F_A$ is continuous on $[t_l,t_r)$ there exists some $t^\ast \in [t_l,t_r)$ such that
\begin{align*}
  0 < y -\varepsilon \leq F_A(t) \leq y \qquad \text{ for all } t \in [t^\ast,t_r)
\end{align*}
where the first inequality follows from the fact that $y >  \delta y/ (\delta + y)$.
\begin{figure}[ht!]
        \begin{center}
            \begin{tikzpicture}[scale=4, domain=0:1]
	\draw[->,line width=1] (-0.2,0) -- (1.2,0);
	\draw[->,line width=1] (0,-0.2) -- (0,1.2);
	\draw[-,line width=1] (0,1) -- (1,1);
	\draw[-,line width=1] (1,0) -- (1,1);
	\draw[-,line width=1] (1,0) -- (1,-0.02);
	\node[below=1pt of {(1,-0.02)}, scale= 0.75, outer sep=2pt,fill=white] {$1$};
	\draw[-,line width=1] (0,1) -- (-0.02,1);
	\node[left=1pt of {(-0.02,1)}, scale= 0.75, outer sep=2pt,fill=white] {$1$};
	\draw[scale = 1, color=blue]   plot[samples=100] (\x,{\x^0.5});
	\draw[scale = 1, color=blue]   plot[samples=100] (\x,{\x^3.5});
	\draw[scale = 1, color=blue]   plot[samples=100] (\x,{\x^2});
	\draw[-,line width=1] (0.5,0) -- (0.5,-0.02);
	\node[below=1pt of {(0.5,-0.02)}, scale= 0.75, outer sep=2pt,fill=white] {$u_1$};
	\draw[-,line width=1] (0.6,0) -- (0.6,-0.02);
	\node[below=1pt of {(0.6,-0.02)}, scale= 0.75, outer sep=2pt,fill=white] {$u_2$};
	\draw[-,line width=1] (0.7,0) -- (0.7,-0.02);
	\node[below=1pt of {(0.7,-0.02)}, scale= 0.75, outer sep=2pt,fill=white] {$u^\ast$};
	\draw[-,line width=1] (0,0.7) -- (-0.02,0.7);
	\node[left=1pt of {(-0.02,0.7)}, scale= 0.75, outer sep=2pt,fill=white] {$v_2$};
	\draw[-,line width=1] (0,0.6) -- (-0.02,0.6);
	\node[left=1pt of {(-0.02,0.6)}, scale= 0.75, outer sep=2pt,fill=white] {$v_1$};
	\draw[-,line width=1] (0,0.49) -- (-0.02,0.49);
	\node[left=1pt of {(-0.02,0.49)}, scale= 0.75, outer sep=2pt,fill=white] {$v^\ast$};
	\draw[-,line width=1, dashed] (0,0.7) -- (0.6,0.7); 
	\draw[-,line width=1, dashed] (0,0.6) -- (0.6,0.6); 
	\draw[-,line width=1, dashed] (0,0.49) -- (0.7,0.49); 
	\draw[-,line width=1, dashed] (0.5,0) -- (0.5,0.7); 
	\draw[-,line width=1, dashed] (0.6,0) -- (0.6,0.7); 
	\draw[-,line width=1, dashed] (0.7,0) -- (0.7,0.49); 
	\node at (0.5,0.6)[circle,fill,inner sep=1.5pt, red]{};
	\node at (0.6,0.6)[circle,fill,inner sep=1.5pt, red]{};
	\node at (0.5,0.7)[circle,fill,inner sep=1.5pt, red]{};
	\node at (0.6,0.7)[circle,fill,inner sep=1.5pt, red]{};
	\node at (0.7,0.49)[circle,fill,inner sep=1.5pt, green]{};
	\node[right=2pt of {(0.8,0.45)}, scale= 0.8, outer sep=2pt,fill=white] {$f_{t_l}$};
	\node[above=2pt of {(0.45,0.72)}, scale= 0.8, outer sep=2pt,fill=white] {$f_{t_r}$};
	\node[above=2pt of {(0.7,0.57)}, scale= 0.8, outer sep=2pt,fill=white] {$f_{t^\ast}$};
	\end{tikzpicture}
	\caption{Blue lines are the contour lines of $f_{t_r}, f_{t^\ast}$ and $f_{t_l}$ where $t_l < t^\ast < t_r$.}\label{FigureAppendix1}
        \end{center}
\end{figure}
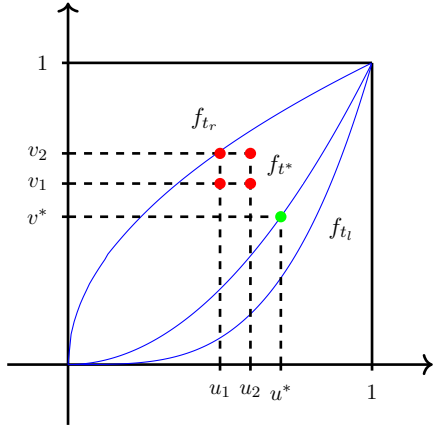
Now, 
let $u_1 \in (0,1)$ and set $v_2 := f_{t_r}(u_1) = u_1^{\frac{1}{t_r}-1} \in (0,1)$.
Then we immediately obtain $h(u_1,v_2) = t_r$ and monotonicity of $h$ implies
$$
        h(u,v) < \min\{h(u,v_2), h(u_1,v)\} \leq \max\{h(u,v_2), h(u_1,v)\} < h(u_1,v_2) = t_r
$$
for all $u \in (u_1,1)$ and all $v \in (0,v_2)$.
Since h is continuous and has convex or concave contour lines respectively, 
there exist $u^\ast \in (u_1,1)$ as well as $v^\ast \in (0, v_2)$ such that
$h(u,v) \geq t^\ast$ for all $(u,v) \in [u_1,u^\ast]\times[v^\ast,v_2]$ and hence
$$
  0 < y -\varepsilon \leq F_A(h(u,v)) \leq y \qquad \text{ for all } (u,v) \in ([u_1,u^\ast]\times[v^\ast,v_2])\backslash \{(u_1,v_2)\}.
$$
We therefore have
\begin{align*}
  \sup\limits_{\substack{u \in (u_1,u^\ast) \\ v \in (v^\ast,v_2)}} \frac{F_A(h(u_1,v)) F_A(h(u,v_2))}{F_A(h(u_1,v_2)) F_A(h(u,v))} 
	&\leq \frac{y^2}{(\delta+y)(y-\varepsilon)} 
	  \\&<  \frac{y^2}{(\delta+y)(y-\frac{\delta y}{ \delta + y})} 
		=  \frac{y^2}{y(\delta+y) - \delta y} = 1
\end{align*}
Setting $\beta:= y^2/[(y+\delta)(y-\varepsilon)] < 1$ we then obtain
$$
  \sup\limits_{\substack{u \in (u_1,u^\ast) \\ v \in (v^\ast,v_2)}} \frac{F_A(h(u_1,v)) F_A(h(u,v_2))}{F_A(h(u_1,v_2)) F_A(h(u,v))} 
	\leq \beta
$$
and continuity of copulas implies the existence of some $u_2 \in (u_1,u^\ast)$ and some $v_1 \in (v^\ast,v_2)$ such that $C(u_2,v_1) > \beta \cdot C(u_2,v_2)$.
\\
We are now in the position to show that inequality \eqref{MKTP2.Def} for the chosen rectangle $[u_1,u_2] \times [v_1,v_2]$ fails to hold.
First note that $K_C(u_2,[0,v_1])= C(u_2,v_1)/u_2 \cdot F_A(h(u_2,v_1))>0$, hence
\begin{eqnarray*}
  \frac{K_C(u_1,[0,v_1]) \, K_C(u_2,[0,v_2])}{K_C(u_1,[0,v_2]) \, K_C(u_2,[0,v_1])}
	&   =  & \frac{C(u_1,v_1)F_A(h(u_1,v_1)) \, C(u_2,v_2)F_A(h(u_2,v_2))}{C(u_1,v_2)F_A(h(u_1,v_2)) \, C(u_2,v_1)F_A(h(u_2,v_1))}
	\\*
	& \leq & \frac{C(u_1,v_1)}{C(u_1,v_2)} \, \beta \, \frac{C(u_2,v_2)}{C(u_2,v_1)}
	\\*
	&   <  & \frac{C(u_1,v_1)}{C(u_1,v_2)}
	\\*
	& \leq & 1
\end{eqnarray*}
This proves the assertion.
\end{proof}

Next, consider some Pickands dependence function $A$ with $A \neq 1$.
Since $A$ is convex it attains its minimum on $(0,1)$, 
i.e. there exists some $\delta \in (0,\infty)$ such that 
$$
  \delta 
	= \argmin_{\gamma \in (0,\infty)} A\left(\frac{1}{1+\gamma}\right)
$$
Since there may exist several minimizer we set
$$
 \beta_A 
:= \sup \Big\{ \delta\in(0,\infty): \delta = \argmin\limits_{\gamma \in (0,\infty)} A\left(\frac{1}{1+\gamma}\right) \Big\} 
$$
Lemma \ref{LemmaCase2Continuous} simplifies the proof of Lemma \ref{NecessaryCondition} in the continuous case.

\begin{Lemma}\label{LemmaCase2Continuous}
Suppose that $D^+A$ is continuous with $A \neq 1$ and define the mapping $g: [0,\beta_A] \to \R$ by
$$
  g(\alpha):= (1+\alpha)A\left(\frac{1}{1+\alpha}\right) - \alpha
$$
Then
\begin{enumerate} \itemsep=0mm
\item[(1)]
$g$ is strictly decreasing on $(0,\beta_A)$.

\item[(2)]
$g(0) = 1$ and $g(\alpha) \geq 0$ for all $\alpha \in (0,\beta_A)$.

\item[(3)]
the inequality 
$$
  F_A\left(\frac{1}{1+\beta_A}\right) 
	< F_A\left(\frac{1}{1+\alpha}\right)
$$
holds for all $\alpha \in(0,\beta_A)$.

\item[(4)]
the inequality 
$$
  \frac{\log\left(\frac{F_A\left(\frac{1}{1+\beta_A}\right)}{F_A\left(\frac{1}{1+\alpha}\right)}\right)}{g(\alpha) -g(\beta_A)} < 0
$$
holds for all $\alpha \in(0,\beta_A)$.
\end{enumerate}
\end{Lemma}
\begin{proof}
First of all, the derivative of $g$ on $(0,\beta_A)$ fulfills
$$
  g'(\alpha) 
	= \underbrace{\left( A\left(\frac{1}{1+\alpha}\right)-1\right)}_{< 0} - \underbrace{\frac{D^+A\left(\frac{1}{1+\alpha}\right)}{1+\alpha}}_{\geq 0}
	< 0
$$
for all $\alpha \in (0,\beta_A)$ which implies that $g$ is strictly decreasing on $(0,\beta_A)$.
This proves (1), and 
(2) is immediate from the fact that $A(t) \geq 1-t$ for all $t \in \I$.
Now, we prove (3) and (4).
Continuity of $D^+A$ implies
\begin{align*}
  F_A\left(\frac{1}{1+\beta_A}\right) 
	& = A\left(\frac{1}{1+\beta_A}\right) + \left(1-\frac{1}{1+\beta_A}\right) D^+A\left(\frac{1}{1+\beta_A}\right) 
	\\
	& = A\left(\frac{1}{1+\beta_A}\right) 
	\\
	& < A\left(\frac{1}{1+\alpha}\right) + \left(1-\frac{1}{1+\alpha}\right) D^+A \left(\frac{1}{1+\alpha}\right) 
	\\
	& = F_A\left(\frac{1}{1+\alpha}\right)
\end{align*}
for all $\alpha \in (0,\beta_A)$. 
Since $F_A \big(1/(1+\beta_A)\big) = A\big(1/(1+\beta_A)\big) > 0 $ we therefore obtain 
$\log \big(F_A\big(1/(1+\beta_A)\big)\big) < \log \big( F_A\big(1/(1+\alpha)\big)\big)$ 
and hence 
$$
  \frac{\log\left(F_A\left(\frac{1}{1+\beta_A}\right)\right) - \log\left(F_A\left(\frac{1}{1+\alpha}\right)\right)}{g(\alpha) -g(\beta_A)} < 0
$$
This proves the assertion.
\end{proof}

Lemma \ref{LemmaFAConstantNotMKTP2} provides another structure for $F_A$ that contradicts the MK-TP2 property.
Its proof is given below and employs the results presented in Lemma \ref{LemmaLinearNotTP2Equality} and Lemma \ref{LemmaLinearNotTP2Inequality}.

\begin{Lemma}\label{LemmaLinearNotTP2Equality}
For every $a\in\R$ the identity
\begin{align*}
  u_1^{a(h(u_1,v_1)-h(u_1,v_2))} v_1^{a(h(u_1,v_1)-h(u_2,v_1))} u_2^{a(h(u_2,v_2)-h(u_2,v_1))} v_2^{a(h(u_2,v_2)-h(u_1,v_2))} = 1
\end{align*}
holds for all $0 < u_1 \leq u_2 < 1$ and all $0 < v_1 \leq v_2 < 1$.
\end{Lemma}

\begin{proof}
For simplicity we use the following abbreviation: $h_{ij} := h(u_i,v_j)$. 
W.l.o.g. assume $a=1$. 
Then straightforward calculation yields
\begin{eqnarray*}
  \lefteqn{\log \big( u_1^{h_{11}-h_{12}} v_1^{h_{11}-h_{21}} u_2^{h_{22}-h_{21}} v_2^{h_{22}-h_{12}} \big)}
	\\
	& = & (h_{11} - h_{12}) \log(u_1) + (h_{11} - h_{21}) \log(v_1) + (h_{22} - h_{21}) \log(u_2) \\&&+\, ( h_{22} - h_{12}) \log(v_2) 
	\\
	& = & h_{11}(\log(u_1) + \log(v_1)) - h_{12} (\log(u_1) + \log(v_2)) 
	\\&&-\, h_{21} (\log(u_2)) + \log(v_1)) + h_{22} (\log(u_2)) + \log(v_2)) 
	\\
	& = & \log(u_1) - \log(u_1) - \log(u_2) + \log(u_2) 
	\\
	& = & 0
\end{eqnarray*}
which proves the result.
\end{proof}

\begin{Lemma}\label{LemmaLinearNotTP2Inequality}
Suppose that $0< t_1 < t_2 < s < 1$ such that
\begin{align*}
  \frac{1-s}{s} \frac{t_2}{1-t_2} > \frac{1-t_2}{t_2} \frac{t_1}{1-t_1}
\end{align*}
Then the inequality 
$$h(u_2,v_1) = t_1 < \min\{h(u_1,v_1), h(u_2,v_2)\} \leq \max\{h(u_1,v_1), h(u_2,v_2)\} \leq t_2 < s = h(u_1,v_2)$$
holds for all $u_1,u_2 \in(0,1)$ such that 
$u_1^{\frac{1-s}{s} \frac{t_2}{1-t_2}} < u_2 < u_1^{\frac{1-t_2}{t_2} \frac{t_1}{1-t_1}}$, $v_1 = f_{t_1}(u_2)$ and $v_2 = f_{s}(u_1)$.
\end{Lemma}
\begin{proof} 
For simplicity we use the following abbreviation: $h_{ij} := h(u_i,v_j)$. 
By assumption, 
$u_1^{\frac{1-s}{s} \frac{t_2}{1-t_2}} < u_1^{\frac{1-t_2}{t_2} \frac{t_1}{1-t_1}}$ which allows choosing $u_2$ in between. 
Since $\frac{1-s}{s} \frac{t_2}{1-t_2} < 1$ we first have $u_1 < u_2$.
We further show that $v_1 < v_2$ which can be seen as follows: since $\frac{1-s}{s} \frac{t_1}{1-t_1} < \frac{1-t_2}{t_2} \frac{t_1}{1-t_1}$ we obtain
$ u_2 < u_1^{\frac{1-t_2}{t_2} \frac{t_1}{1-t_1}} < u_1^{\frac{1-s}{s} \frac{t_1}{1-t_1}} $
or, equivalently,
$ v_1 
	= f_{t_1}(u_2) 
	= u_2^{\frac{1-t_1}{t_1}} 
	< u_1^{\frac{1-s}{s}} 
	= f_{s}(u_1) 
	= v_2 $.
Direct computation then yields 
$h_{21}= t_1$ and $h_{12} = s$, and monotonicity of $h$ implies
\begin{align*}
  h_{21} 
	  < h_{22} 
	  =  \frac{\log(u_2)}{\log(u_2v_2)} 
	  <  \frac{\log(u_1^{\frac{1-s}{s} \frac{t_2}{1-t_2}})}{\log(u_1^{\frac{1-s}{s} \frac{t_2}{1-t_2}} u_1^{\frac{1-s}{s}})} 
	  =  \frac{\frac{1-s}{s} \frac{t_2}{1-t_2}}{\frac{1-s}{s}(\frac{t_2}{1-t_2} + 1)} 
		= t_2
		< s
\end{align*}
and
\begin{align*}
  h_{21} 
	  <  h_{11} 
	  =  \frac{\log(u_1)}{\log(u_1v_1)} 
		< \frac{\log(u_2^{\frac{1-t_1}{t_1} \frac{t_2}{1-t_2}})}{\log(u_2^{\frac{1-t_1}{t_1} \frac{t_2}{1-t_2}} u_2^{\frac{1-t_1}{t_1}})} 
		= \frac{\frac{1-t_1}{t_1} \frac{t_2}{1-t_2}}{\frac{1-t_1}{t_1} (\frac{t_2}{1-t_2} + 1)} 
    = t_2
		< s
\end{align*}
from which the assertion hence follows.
\end{proof}

We are now in the position to prove Lemma \ref{LemmaFAConstantNotMKTP2}:

\begin{proof}[\textbf{of Lemma \ref{LemmaFAConstantNotMKTP2}}]
For simplicity we use the following abbreviation: $h_{ij} := h(u_i,v_j)$. 
By assumption, we have $F_A(t) > 0$ for all $t \in [t_1, 1]$. 
If $F_A$ has a discontinuity point in $[t_1,1)$ then Lemma \ref{LemmaCase2Discontinuous} implies that the corresponding EVC is not MK-TP2. 
Therefore, it remains to prove the result for $F_A$ continuous on $[t_1,1)$.

The solution of the first-order differential equation $F_A(t) = A(t) + (1-t) D^+A(t) = c$ on $[t_1,t_2]$ is a linear function (see, e.g., \cite{walter2013gewohnliche}),
i.e., there exist $a,b \in \R$ such that $A(t) = at + b$ on $[t_1,t_2]$,
hence $a+b = at + b + (1-t) a = F_A(t) = c \in (0,1)$ for all $t\in[t_1,t_2]$.
W.l.o.g. set $t_2 := \sup\{ t \in \I \, : \, F_A(t) \leq c\}$. Then $F_A(t) > c = a + b$  for all $t \in (t_2,1]$.

In what follows we choose a rectangle $u_1 \leq u_2$ and $v_1 \leq v_2$ such that the corresponding values fulfill
$h_{21}, h_{11}, h_{22} \in [t_1,t_2]$, $h_{12} > t_2$ 
and hence $F_A(h_{21}) = F_A(h_{11}) = F_A(h_{22}) = c < F_A(h_{12})$ 
(see Figure \ref{FigureAppendix2} for an illustration), which we then use to show that $C$ is not MK-TP2.
\begin{figure}[ht!]
        \begin{center}
            \begin{tikzpicture}[scale=4, domain=0:1]
	\draw[->,line width=1] (-0.2,0) -- (1.2,0);
	\draw[->,line width=1] (0,-0.2) -- (0,1.2);
	\draw[-,line width=1] (1,0) -- (1,-0.02);
	\node[below=1pt of {(1,-0.02)}, scale= 0.75, outer sep=2pt,fill=white] {$1$};
	\draw[-,line width=1] (0,1) -- (-0.02,1);
	\node[left=1pt of {(-0.02,1)}, scale= 0.75, outer sep=2pt,fill=white] {$1$};
	\draw[-,line width=1] (0.25,0) -- (0.25,-0.02);
	\node[left=1pt of {(0.27,-0.05)}, scale= 0.75, outer sep=2pt] {$t_1$};
	\draw[-,line width=1] (0.75,0) -- (0.75,-0.02);
	\node[right=1pt of {(0.73,-0.05)}, scale= 0.75, outer sep=2pt] {$t_2$};
	\draw[-,line width=1, green] (0.85,0) -- (0.85,-0.02);
	\node[right=1pt of {(0.83,-0.05)}, scale= 0.75, outer sep=2pt] {\textcolor{green}{$h_{21}$}};
	\draw[thick, blue, rounded corners=2mm] (0,0)--(0.21,0.6)--(0.79,0.6)--(1,1);
	\draw[decorate,decoration={brace,amplitude=10pt, mirror}, xshift=0pt,yshift=0pt]
	(0.25,0) -- (0.75,0) node [black,midway,xshift=-0cm, yshift=-0.5cm] 
	{\tiny{$\in h_{21}, h_{11}, h_{22}$} };
	\draw[-,line width=1, dashed] (0.25,0) -- (0.25,0.6); 
	\draw[-,line width=1, dashed] (0.75,0) -- (0.75,0.6); 
	\draw[-,line width=1, green, dashed] (0.85,0) -- (0.85,0.72); 
	\end{tikzpicture}
	\caption{Example of $F_A$ which is constant on $[t_1,t_2]$.}\label{FigureAppendix2}
        \end{center}
\end{figure}
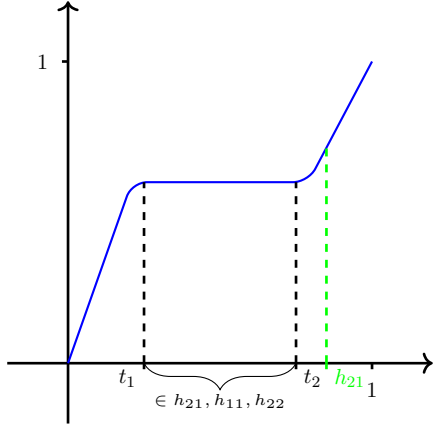
\\
Since $t \mapsto (1-t)/t$ is strictly decreasing and continuous on $(0,1)$,
we have $(1-t_2)/t_2 \cdot t_1/(1-t_1) < 1$ and there exists some $s_1 \in (t_2,1)$ such that 
$$
  \frac{1-s}{s} \frac{t_2}{1-t_2} > \frac{1-t_2}{t_2} \frac{t_1}{1-t_1}
$$
for all $s \in (t_2,s_1]$.
Moreover, since $b \in [0,1]$ and $c\in (0,1)$ we have $a = c-b \in (-1,1)$ such that the Pickands dependence function $A$ does not coincide with the lower bound
$t \mapsto \max \lbrace 1-t,t\rbrace$ on $(t_2,1]$
(otherwise $D^+A$ would have a discontinuity point which then would contradict Lemma \ref{LemmaCase2Discontinuous}).
Therefore we can extend the linear part of $A$ to a certain degree and are still above the lower bound (see Figure \ref{FigureAppendix3}).
Formally this means that there exists some $s_2 \in(t_2,1)$ such that $a t + b > \max \lbrace 1-t,t\rbrace \geq 1/2$ for all $t \in(t_2,s_2]$.
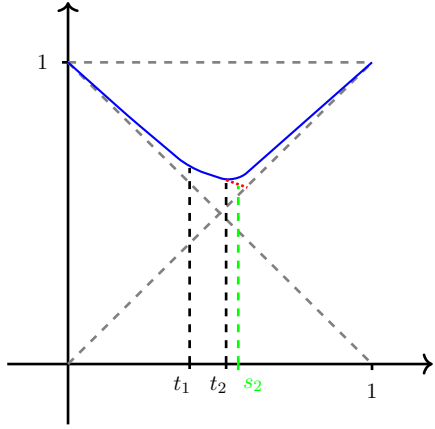
\begin{figure}[ht!]
        \begin{center}
            \begin{tikzpicture}[scale=4, domain=0:1]
	\draw[->,line width=1] (-0.2,0) -- (1.2,0);
	\draw[->,line width=1] (0,-0.2) -- (0,1.2);
	\draw[-,line width=1] (1,0) -- (1,-0.02);
	\node[below=1pt of {(1,-0.02)}, scale= 0.75, outer sep=2pt,fill=white] {$1$};
	\draw[-,line width=1] (0,1) -- (-0.02,1);
	\node[left=1pt of {(-0.02,1)}, scale= 0.75, outer sep=2pt,fill=white] {$1$};
	\draw[-,line width=1, dashed, gray] (0,0) -- (1,1); 
	\draw[-,line width=1, dashed, gray] (0,1) -- (1,0); 
	\draw[-,line width=1, dashed, gray] (0,1) -- (1,1); 
	\draw[-,line width=1] (0.4,0) -- (0.4,-0.02);
	\node[left=1pt of {(0.45,-0.07)}, scale= 0.75, outer sep=2pt,fill=white] {$t_1$};
	\draw[-,line width=1] (0.52,0) -- (0.52,-0.02);
	\node[left=1pt of {(0.57,-0.07)}, scale= 0.75, outer sep=2pt,fill=white] {$t_2$};
	\draw[-,line width=1, green] (0.56,0) -- (0.56,-0.02);
	\node[right=1pt of {(0.53,-0.07)}, scale= 0.75, outer sep=2pt,fill=white] {\textcolor{green}{$s_2$}};
	\draw[thick, blue, rounded corners=2mm] (0,1)--(0.2,0.82)--(0.4,0.65)--(0.55,0.6)--(1,1);
	\draw[-,line width=1, dashed] (0.4,0) -- (0.4,0.65); 
	\draw[-,line width=1, dashed] (0.52,0) -- (0.52,0.6); 
	\draw[-,line width=1, green, dashed] (0.56,0) -- (0.56,0.59); 
	\draw[-,line width=1, red, densely dotted] (0.52,0.61) -- (0.59,0.585); 
	\end{tikzpicture}
	\caption{Example of Pickands dependence function which is linear on $[t_1,t_2]$ with the described extension (red line).}\label{FigureAppendix3}
        \end{center}
\end{figure}
Setting $s := (t_2+  \min\{s_1,s_2\})/2$, we then have
\begin{itemize}
  \item $F_A(s) > a+b$
  \item $ \frac{1-s}{s} \frac{t_2}{1-t_2} > \frac{1-t_2}{t_2} \frac{t_1}{1-t_1}$
  \item $A(s) - (as+b) \leq 1 - \frac{1}{2} = \frac{1}{2}$
\end{itemize}
Now, choose $u_1 \in (0,1)$ large enough such that
\begin{align}\label{LinearNotTP2Assumption}
  a+b < u_1^{\frac{1}{2s}} F_A(s) \leq \big(u_1^{\frac{1}{s}}\big)^{A(s) - (as+b)} F_A(s) 
\end{align}
By choosing $u_2 \in(0,1)$ such that $u_1^{\frac{1-s}{s} \frac{t_2}{1-t_2}} < u_2 < u_1^{\frac{1-t_2}{t_2} \frac{t_1}{1-t_1}}$ and setting $v_1 = f_{t_1}(u_2)$ and $v_2 = f_{s}(u_1)$, 
applying Lemma \ref{LemmaLinearNotTP2Inequality} yields
$$
  h_{21} = t_1 < \min(h_{11},h_{22}) \leq \max(h_{11},h_{22}) \leq t_2 < s = h_{12}
$$
For the Markov kernels, we then obtain
\begin{align}\label{LinearNotTPProofIneq1}
  K_A(u_1,[0,v_1]) \, K_A(u_2,[0,v_2]) 
	& = \frac{C(u_1,v_1)}{u_1}F_A(h_{11}) \frac{C(u_2,v_2)}{u_2}F_A(h_{22}) \notag
	\\
	& = \frac{(a+b)^2}{u_1 u_2}
        \left(u_1 v_1\right)^{ah_{11}+b}\left(u_2 v_2\right)^{ah_{22}+b} 
\end{align}
and
\begin{align}\label{LinearNotTPProofIneq2}
  K_A(u_1,[0,v_2]) \, K_A(u_2,[0,v_1]) 
	& = \frac{C(u_1,v_2)}{u_1}F_A(h_{12}) \frac{C(u_2,v_1)}{u_2}F_A(h_{21}) \notag
	\\
	& = \frac{a+b}{u_1 u_2} F_A(s)
         \left(u_1 v_2\right)^{A(s)} \left(u_2 v_1\right)^{ah_{21}+b}
\end{align}
and Lemma \ref{LemmaLinearNotTP2Equality} yields 
$v_1^{a(h_{11}-h_{21})} u_2^{a(h_{22}-h_{21})} 
 = u_1^{a(h_{12}-h_{11})} v_2^{a(h_{12}-h_{22})}$ which gives
\begin{eqnarray*}
  \frac{K_A(u_1,[0,v_1]) \, K_A(u_2,[0,v_2])}{K_A(u_1,[0,v_2]) \, K_A(u_2,[0,v_1])}
	& = & \frac{(a+b)}{F_A(s)}
				\frac{\left(u_1 v_1\right)^{ah_{11}+b}\left(u_2 v_2\right)^{ah_{22}+b}}
				     {\left(u_1 v_2\right)^{A(s)} \left(u_2 v_1\right)^{ah_{21}+b}} 
	\\
	& = & \frac{(a+b)}{F_A(s)}
				\big(u_1^{ah_{11}+b-A(s)} v_1^{a(h_{11}-h_{21})} u_2^{a(h_{22}-h_{21})} v_2^{ah_{22}+b-A(s)}\big)
	\\
	& = & \frac{(a+b)}{F_A(s)}
				\big(u_1^{ah_{11}+b-A(s)} u_1^{a(h_{12}-h_{11})} v_2^{a(h_{12}-h_{22})} v_2^{ah_{22}+b-A(s)}\big)
	\\
	& = & \frac{(a+b)}{F_A(s)} \, (u_1v_2)^{ah_{12}+b-A(s)}
	\\
	& = & \frac{(a+b)}{F_A(s)} \, \big(u_1^{\frac{1}{s}}\big)^{as+b-A(s)} < 1
\end{eqnarray*}
Therefore, the EVC is not MK-TP2.
\end{proof}

\end{document}